\documentclass[twoside,10pt,a4paper]{article}
\usepackage[colorlinks=true]{hyperref}                        
\hypersetup{ linkcolor=blue,  filecolor=black,  urlcolor=red,  citecolor=black}

\usepackage{amsmath,latexsym,amsfonts,indentfirst,amsthm,amsxtra,amssymb,bm,stmaryrd,setspace,tikz,mathtools,extarrows,graphicx}
\usepackage{mathrsfs}
\usepackage[numbers,sort&compress]{natbib}
\usepackage[perpage,symbol*]{footmisc}
\usepackage{relsize}
\usepackage{xcolor}
\usepackage{BOONDOX-cal}
\usepackage{anyfontsize}
\usepackage{authblk}

\numberwithin{equation}{section}

\makeatletter

\newcommand{\Rmnum}[1]{\expandafter\@slowromancap\romannumeral #1@}
\makeatother
\newtheorem{lemma}{Lemma}[section]
\newtheorem{proposition}[lemma]{Proposition}
\newtheorem{theorem}{Theorem}[section]

\theoremstyle{remark}
\newtheorem{remark}{Remark}[section]
\newcommand{\vertiii}[1]{{\left\vert\kern-0.25ex\left\vert\kern-0.25ex\left\vert #1
		\right\vert\kern-0.25ex\right\vert\kern-0.25ex\right\vert}}

\usepackage{accents}  
\makeatletter
\def\wideubar{\underaccent{{\cc@style\underline{\mskip10mu}}}}
\def\Wideubar{\underaccent{{\cc@style\underline{\mskip8mu}}}}
\makeatother
\makeatletter
\def\widebar{\accentset{{\cc@style\underline{\mskip10mu}}}}
\def\Widebar{\accentset{{\cc@style\underline{\mskip8mu}}}}
\makeatother

\newcommand{\lrn}{\left\vert\kern-0.3ex\left\vert\kern-0.3ex\left\vert}
\newcommand{\rrn}{\right\vert\kern-0.3ex\right\vert\kern-0.3ex\right\vert}

\newcommand{\bemph}[1]{{\upshape#1}} 
\newcommand{\ep}[1]{\bemph{(}#1\bemph{)}} 

\newcommand*\circled[1]{\tikz[baseline=(char.base)]{
		\node[shape=circle,draw,inner sep=1pt] (char) {#1};}}

\allowdisplaybreaks[3]
\numberwithin{equation}{section}

\begin{document}

	\title{
		\bf  Global Smooth Radially Symmetric Solutions to a Multidimensional Radiation Hydrodynamics Model}
	
	\author[a]{Huijiang Zhao\thanks{Corresponding author.
			Email: hhjjzhao@hotmail.com(H. Zhao); boran\_zhu@outlook.com(B. Zhu)}}
	\author[a]{Boran Zhu}
	\affil[a]{\small School of Mathematics and Statistics, Wuhan University, Wuhan 430072, China}
	\date{\empty}
	\maketitle
	\begin{abstract}
		
		The motion of a compressible inviscid radiative flow can be described by the radiative Euler equations, which consists of the Euler system coupled with a Poisson equation for the radiative heat flux through the energy equation. Although solutions of the compressible Euler system will generally develop singularity no matter how smooth and small the initial data are, it is believed that the radiation effect does imply some dissipative mechanism, which can guarantee the global regularity of the solutions of the radiative Euler equations at least for small initial data.
		
		Such an expectation was rigorously justified for the one-dimensional case, as for the multidimensional case, to the best of our knowledge, no result was available up to now. The main purpose of this paper is to show that the initial-boundary value problem of such a radiative Euler equation in a three-dimensional bounded concentric annular domain does admit a unique global smooth radially symmetric solution provided that the initial data is sufficiently small.	
		\vspace{2mm}
		
		\noindent  {\bf Keywords:} Multidimensional radiative Euler equation;  Global smooth radially symmetric solution; Symmetric hyperbolic system; A priori estimates.

		\vspace{2mm}
	\end{abstract}
	\section{introduction}
	
	In the modeling of astrophysical flows, reentry problems, or high temperature combustion phenomena, we have to deal with high-temperature fluids and when the temperatures of the fluids are more than 10000K, radiative effect should be taken into consideration \ep{\cite{2004Radiation}, \cite{MR0111583}}.
	
	In fact, when fluid interacts with radiation through energy exchange, since the momentum caused by the radiation can be neglected, while the radiative flux must be added into the energy equation since the transport of energy carried by the radiation process is more important, we can then use the following radiative Euler equations, which is a compressible Euler system coupled with an elliptic equation for radiation flux, as an approximate system to describe the motion of radiation hydrodynamics \ep{For more Physical background, we refer to Chapter \Rmnum{11} and \Rmnum{12} in \cite{1965Introduction}, and for derivation of the equation for radiation flux, one can also refer to \cite{MR3403137} and \cite{MR2802993}}:
	\begin{equation}
	\label{cns1}
	\left\{	
	\begin{aligned}
	&\rho_t+\nabla_x\cdot(\rho \bm{u})=0,\\
	&(\rho {\bf u})_t+\nabla_x\cdot(\rho({\bf u}\otimes{\bf u})+\nabla_x P=0,\\
	&\left(\rho\left(e+\frac{|{\bf u}|^2}{2}\right)\right)_t+\nabla_x\cdot \left(\left(\rho\left(e+\frac{|{\bf u}|^2}{2}\right)+P\right){\bf u}\right)+\nabla_x\cdot{\bf q}=0,\\
	&-\nabla_x\left(\nabla_x\cdot{\bf q}\right)+a_1{\bf q}+b_1\nabla_x\left(\theta^4\right)=0.
	\end{aligned}	
	\right.
	\end{equation}
	Here $a_1>0, b_1>0$ are positive constants. The primary dependent variables are the fluid density $\rho$, the fluid velocity ${\bf u}\in \mathbb{R}^3$, the absolute temperature $\theta$, and the radiative heat flux ${\bf q}\in \mathbb{R}^3$. The pressure $P$, the internal energy $e$, and the other three thermodynamic variables the density $\rho$, the absolute temperature $\theta$, and the specific entropy $s$ are related through Gibbs' equation $de=\theta ds-Pd\rho^{-1}$.
	
	Throughout this paper, we consider only ideal, polytropic gases:
	\begin{eqnarray}\label{transform}
	P&=&R\rho\theta=A\rho^{\frac{1+C_v}{C_v}}{\rm exp}\left(\frac{s}{C_v}\right),\nonumber \\
	\theta&=&\frac AR\rho^{\frac{1}{C_v}}{\rm exp}\left(\frac{s}{C_v}\right),\\
	\rho&=&\left(\frac AR\right)^{\frac{C_v}{C_v+1}}P^{\frac{C_v}{C_v+1}}{\rm exp}\left(-\frac{s}{C_v+1}\right),\nonumber
	\end{eqnarray}
	where $R>0, A>0$ (specific gas constant) and $C_v>0$ (specific heat at constant volume) are positive
	constants. Without loss of generality, we assume in the rest of this paper that $a_1=b_1=R=1$.

	Due to the high nonlinearities of the rariative Euler equations \eqref{cns1} and the complex physical phenomena described by it, its mathematical theory is one of the hottest topics in the field of nonlinear partial differential equations. The main observation is that, although solutions of the compressible Euler system will generally develop singularity no matter how smooth and how small the initial data are \ep{one can refer to \citet{MR0369934}, \citet{MR0165243}, \citet{MR0815196} and the book \cite{MR3288725}}, it is believed that the radiation effect does imply some dissipative mechanism, which can guarantee the global regularity of the solutions to radiative Euler equations at least for small initial data. One of the main mathematical problems concerning the radiative Euler equations \eqref{cns1} is to justify the above expectation rigorously.
	
	For such a problem, the results for one-dimensional case are quite complete, one can refer to \citet{MR2022134} for the global existence of classical solutions and \citet{MR4085490} for pointwise structure. About the stability of elementary waves, one can refer to \citet{MR2234210} and \cite{MR2371479} respectively for the existence and stability of shock wave, \citet{MR2836843} for rarefaction wave and \citet{MR2812581} for viscous contact wave. As for the composite of several elementary waves, there are \citet{MR3906863} about two viscous shock waves,  \citet{MR2873128} about viscous contact wave and rarefaction waves. There are results with respect to some asymptotic limit of the radiation hydrodynamics towards elementary waves, we refer to \citet{MR3492077}, \citet{MR2800574} and \citet{MR3015674}.
	
	While for the multidimensional case, to the best of our knowledge, no result was available up to now. The main purpose of this paper is to show that the initial-boundary value problem of such a radiative Euler equations in a three-dimensional bounded concentric annular domain does admit a unique global smooth radially symmetric solution provided that the initial data is sufficiently small.
	
	Now we transfer the system in our case to a desired form. Since all thermodynamics variables $\rho,\ \theta,\ e,\ P$ as well as the entropy $s$ can be written as functions of any two of them, if we take $P$ and $s$ as independent variables,
	the system \eqref{cns1} can be rewritten in terms of $P$, $\bm{u}$, $s$ and $\bm{q}$ as
	\begin{equation}
	\label{cns2}
	\left\{	
	\begin{aligned}
	&P_t+\left(\bm{u}\cdot\nabla\right)P+\frac{C_v}{C_v+1}P{\rm div}\bm{u}+\frac{1}{C_v}{\rm div}\bm{q}=0,\\
	&\bm{u}_t+(\bm{u}\cdot\nabla)\bm{u}+\frac{\nabla P}{\rho}=0,\\
	&s_t+(\bm{u}\cdot\nabla)s+\frac{{\rm div}\bm{q}}{P}=0,\\
	&-\nabla{\rm div}\bm{q}+\bm{q}+\frac{4\theta^3}{C_v+1}\left(\frac{\nabla P}{\rho}+\theta\nabla s\right)=0.
	\end{aligned}	
	\right.
	\end{equation}
	Here to deduce the last equation in \eqref{cns2}, we have used the fact that
	\begin{align*}
	-\nabla{\rm div}\bm{q}+\bm{q}+4\theta^3\nabla\theta
	&=-\nabla{\rm div}\bm{q}+\bm{q}+4\theta^3\left(\frac{1}{C_v+1}\nabla\theta+\frac{C_v}{C_v+1}\nabla\theta\right)\\
	&=-\nabla{\rm div}\bm{q}+\bm{q}+4\theta^3\left\{\frac{1}{C_v+1}\nabla\theta+\frac{C_v}{C_v+1}\nabla\left[A\rho^{\frac{1}{C_v}}{\rm exp}\left(\frac{s}{C_v}\right)\right]\right\}\\
	&=-\nabla{\rm div}\bm{q}+\bm{q}+4\theta^3\left(\frac{1}{C_v+1}\nabla\theta+\frac{1}{C_v+1}\frac{\theta}{\rho}\nabla\rho+\theta\nabla s\right)\\
	&=-\nabla{\rm div}\bm{q}+\bm{q}+\frac{4\theta^3}{C_v+1}\left(\frac{\nabla P}{\rho}+\theta\nabla s\right)=0.
	\end{align*}
	
	We consider the system \eqref{cns2} in a bounded concentric annular domain $\Omega=\{\bm{x}\ \big|\ \bm{x}\in\mathbb{R}^3,\ |\bm{x}|=r,\ 0<a<r<b<+\infty\}$ with prescribed initial data
	\begin{align} \label{initial1}
	\big(P(0,\bm{x}),\bm{u}(0,\bm{x}), s(0,\bm{x})\big)
	=(P_0(\bm{x}),\bm{u}_0(\bm{x}),s_0(\bm{x})) \quad {\rm for}\  \ \bm{x}\in\Omega
	\end{align}
	and boundary conditions
	\begin{equation}
	\label{bdy1}
	\bm{u}(t,\bm{x})=0,\quad\bm{q}(t,\bm{x})=0\quad {\rm on}\  (t,\bm{x})\in[0,+\infty)\times \partial\Omega.
	\end{equation}
	
	Now we study the spherically symmetric classical solution to the initial boundary problem \eqref{cns2}--\eqref{bdy1} with the form that
	\begin{align*}
	&P(t,\bm{x})=\rho(t,\bm{x})\theta(t,\bm{x})=\tilde{\rho}(t,r)\tilde{\theta}(t,r)=\tilde{P}(t,r),\\
	&\bm{u}(t,\bm{x})=\frac{\tilde{u}(t,r)}{r}\bm{x},\  s(t,\bm{x})=\tilde{s}(t,r),\  \bm{q}(t,\bm{x})=\frac{\tilde{q}(t,r)}{r}\bm{x}.
	\end{align*}
	The corresponding spherical form of the initial boundary value problem \eqref{cns2}--\eqref{bdy1} reads
	\begin{equation}
	\label{cns3}
	\left\{	
	\begin{aligned}
	&\tilde{P}_t+\tilde{u}\tilde{P}_r+\frac{C_v+1}{C_v}\tilde{P}\frac{(r^2\tilde{u})_r}{r^2}+\frac{1}{C_v}\frac{(r^2\tilde{q})_r}{r^2}=0,\\
	&\tilde{u}_t+\tilde{u}\tilde{u}_r+\frac{\nabla \tilde{P}}{\tilde{\rho}}=0,\\
	&\tilde{s}_t+\tilde{u}\tilde{s}_r+\frac{(r^2\tilde{q})_r}{r^2\tilde{P}}=0,\\
	&-\left[\frac{(r^2\tilde{q})_r}{r^2}\right]_r+\tilde{q} +\frac{4\tilde{\theta}^3}{C_v+1}\left(\frac{\tilde{P}_r}{\tilde{\rho}}+\tilde{\theta}\tilde{s}_r\right)=0
	\end{aligned}	
	\right.
	\end{equation}
	with the corresponding initial and boundary conditions
	\begin{alignat}{2}
	\label{initial2}
	\big(\tilde{P}(0,r),\tilde{u}(0,r),\tilde{s}(0,r)\big)&=
	(\tilde{P}_0,\tilde{u}_0,\tilde{s}_0)(r), \qquad && a\leq r\leq b,\\[1mm]
	\label{bdy2}
	(\tilde{u}(t,a),\tilde{q}(t,a))&=(\tilde{u}(t,b),\tilde{q}(t,b))=0,\qquad && t\geq 0.
	\end{alignat}
	
	Now we interpret the initial boundary value problem \eqref{cns2}--\eqref{bdy2} into Lagrangian coordinates.
	We now use $\eqref{cns1}_1$ to find another set of variables $(t',x)$ such that
	\begin{equation*}
	\left\{
	\begin{aligned}
	r&=r(t',x)\\
	t&=t'
	\end{aligned}
	\right.
	\end{equation*}
	and $x$ is not dependent on $t$, which is called Lagrangian coordinate. Define coordinate transformation that
	\begin{equation}\label{coordinate}
	\left\{
	\begin{aligned}
	&r(t',x)=r_0(x)+\int_0^{t'} \tilde{u}(\tau,r(\tau,x))\mathrm{d}\tau,\\
	&t=t',
	\end{aligned}
	\right.
	\end{equation}
	where $r_0(x)$ is defined by following way: First, define the function $h(z)$ as
	\begin{align*}
	h(z):=\int_{a}^{z}y^2\tilde{\rho}_0(t,y)\ \mathrm{d}y.
	\end{align*}
	Since that $\tilde{\rho}_0(y)>0$ (which is assumed in Theorem \ref{thm}), function $h(z)$ is invertible and $h^{-1}(\cdot)$ is well defined. Then, $r_0(x)$ is defined as
	\begin{align}\label{r0}
	r_0(x)=h^{-1}(x).
	\end{align}
	
	We claim that the coordinate $(t',x)$ is Lagrangian coordinate.
	The first thing we want to show is that $x$ is truly not dependent on $t$. So we now prove that $\partial_tx=0$ and
	\begin{align}\label{r}
	x=\int_{a}^{r(t',x)}y^2\tilde{\rho}(t,y)\ \mathrm{d}y=\int_{a}^{r(t,x)}y^2\tilde{\rho}(t,y)\ \mathrm{d}y.
	\end{align}
	From $\eqref{cns1}_1$ and \eqref{coordinate}, we can see that
	\begin{align}\notag
	\partial_t\int_{a}^{r(t,x)}y^2\tilde{\rho}(t,y)\ \mathrm{d}y
	&=r^2\tilde{\rho}(t,r)\partial_tr-\int_{a}^{r(t,x)}y^2\tilde{\rho}_t(t,y)\ \mathrm{d}y\\\notag
	&=r^2\tilde{\rho}(t,x)\partial_tr-\int_{a}^{r(t,x)}\mathrm{d}\left[y^2\tilde{\rho} \tilde{u}(t,y)\right]\\[2.5mm]
	&=r^2\tilde{\rho}\left(\partial_tr-\tilde{u}\right)=0.\label{xt}
	\end{align}
	The second equality holds because after spherically symmetric transform, the equation of $\tilde{\rho}(t,r)$ takes the form that
	\begin{align}\label{mass-conservation}
	\tilde\rho_t+\frac{(r^2 \tilde\rho \tilde{u})_r}{r^2}=0.
	\end{align}
	We integrate \eqref{xt} over $(0,t)$ to get that
	\begin{align*}
	\int_{a}^{r(t,x)}y^2\tilde{\rho}(t,y)\ \mathrm{d}y=\int_{a}^{r_0(x)}y^2\tilde{\rho}_0(y)\ \mathrm{d}y=h(r_0)=x.
	\end{align*}
	Without loss of generality, we set $h(b)=1$, so that $x\in[0,1]$. Since $t'=t$, without confusion, we still use $t$ to denote the time variable and we have
	\begin{align*}
	&\tilde{P}(t,r)=\tilde{\rho}(t,r)\tilde{\theta}(t,r)=\rho(t,x)\theta(t,x)=P(t,x),\ \tilde{u}(t,r)=u(t,x),\ \tilde{s}(t,r)=s(t,x),\  \tilde{q}(t,r)=q(t,x).
	\end{align*}
	The next thing we need to know is that whether this coordinate transformation is well-posed. From the a priori estimates which we are going to derive later , it can be seen that $\rho>0$, so the coordinate transformation is well-posed and identities \eqref{r}, \eqref{coordinate} imply
	\begin{align} \label{r_eq}
	r_t(t,x)=u(t,x),\quad
	r_x(t,x)=\frac{1}{r^2\rho(t,x)}.
	\end{align}
	
	By virtue of \eqref{r_eq},  the system \eqref{cns3} is reformulated to that of $(P,u,s,q)(t,x)$, which takes the form
	\begin{equation}
	\label{cns4}
	\left\{	
	\begin{aligned}
	&P_t+\frac{C_v+1}{C_v}P\rho(r^2u)_x+\frac{1}{C_v}\rho(r^2q)_x=0,\\
	&u_t+r^2P_x=0,\\
	&s_t+\frac{(r^2q)_x}{\theta}=0,\\
	&-r^2\rho\left[\rho(r^2q)_x\right]_x+q+\frac{4r^2\theta^3}{C_v+1}\left( P_x+Ps_x\right)=0,
	\end{aligned}	
	\right.
	\end{equation}
	where $t>0$, $x\in \mathbb{I}:=(0,1)$.
	The corresponding initial and boundary conditions are
	\begin{alignat}{2} \label{initial0}
	(P(0,x),u(0,x),s(0,x))&=
	(P_0(x),u_0(x),s_0(x)),
	\qquad && x\in\mathbb{I},\\[1mm]               \label{bdy}
	(u(t,0),q(t,0))&=(u(t,1),q(t,1))=0,\qquad&& t\geq 0,
	\end{alignat}
	where $(P_0,u_0,s_0):=
	(\tilde{P}_0,\tilde{u}_0,\tilde{s}_0)\circ r_0,$ the symbol $\circ$ denotes composition, and $r_0$ is defined by \eqref{r0}.
	
	To state our main result, we first introduce some notations.
	
Firstly, the initial data $\left(\partial_t^kP(0,x)\right.$,$\partial_t^ku(0,x)$,$\partial_t^ks(0,x)$,$\left.\partial_t^{k-1}q(0,x)\right)$ \ep{for any integer $k\geq 1$} 
are defined inductively by following procedure: first apply $\partial_t^{k-1}$ \ep{$k=1,2,\cdots$} on system \eqref{cns4}; second obtain $\partial_t^{k-1}q(0,x)$ by letting $t=0$ in the last elliptic equation and solving it; Third obtain $\left(\partial_t^kP(0,x),\partial_t^ku(0,x),\partial_t^ks(0,x)\right)$ by solving for $\left(\partial_t^kP(t,x),\partial_t^ku(t,x),\partial_t^ks(t,x)\right)$ in the first three equations and evaluating at $t=0$.

	In order to clarify the space of solutions, we introduce the space
	\begin{align*}
	X_m\left([0,T];\mathbb{I}\right):=\bigcap_{k=0}^{m}C^k\left([0,T];H^{m-k}\left(\mathbb{I}\right)\right)\quad\left(\ \text{for some positive integer } m\ \right)
	\end{align*}
	with norm
	\begin{align*}
	\vertiii{f}_{m,T}=\sup_{0\leq t\leq T} \vertiii{f(t)}_m=\sup_{0\leq t\leq T}\sum_{k=0}^{m}\left\|\partial_t^{k}f(t)\right\|_{H^{m-k}(\mathbb{I})}.
	\end{align*}
	
	According to the above definitions, the value $\vertiii{\left(P,u,s,q,q_x\right)(0)}_m$ is well-defined. Furthermore, if we assume $\|(P_0-1,u_0,s_0-1)\|_{H^m}$ is sufficiently small, it holds that
	\begin{align}\label{XmHm}
	\vertiii{\left(P,u,s,q,q_x\right)(0)}^2_m\leq C_1\left\{\|(P_0-1,u_0,s_0-1)\|^2_{H^m}\right\}\|(P_0-1,u_0,s_0-1)\|^2_{H^m}.
	\end{align}
	Here and in the rest of this paper, we will frequently use $C\{\cdot\}$, $C_i\{\cdot\}$, $G_j\{\cdot\}$ to denote some positive constants which are continuous nondecreasing functions of the quantities listed in braces.
	
	Especially, $\interleave \cdot\interleave_{m,tan}$ is the norm when derivatives with respect to $x$ are not included. That is,
	\begin{align*}
	\vertiii{f(t)}_{m,tan}=\sum_{k=0}^{m}\left\|\partial_t^{k}f(t)\right\|_{L^2 (\mathbb{I})}.
	\end{align*}
	Our ultimate goal is to prove the following theorem concerning the global existence of classical solution to the initial boundary problem \eqref{cns4}--\eqref{bdy} when the initial data is a small perturbation of a constant state which, for simplicity, is taken as (1,0,1). (The corresponding equilibrium state for $(\rho,u,\theta)$ is $(c_\rho,0,c_\theta)$. We can see from \eqref{transform} that $c_\rho$ and $c_\theta$ are some positive constants.)
	\begin{theorem}\label{thm}
		If there exists a small enough constant $\epsilon_0>0$ such that the initial data satisfies
		\begin{equation}\label{initialdata}
		\|(P_0-1,u_0,s_0-1)\|^2_{H^m}\leq\epsilon_0\qquad(m\geq2)
		\end{equation}
		with the corresponding compatibility conditions
		\begin{align}\label{compatible}
		\partial_t^ku(0,x)=0\ \text{on}\ \partial\mathbb{I}\qquad (0\leq k\leq m)
		\end{align}
		Then the initial boundary problem \eqref{cns4}--\eqref{bdy} admits a unique global-in-time solution satisfying
		\begin{eqnarray*}
			(P(t,x)-1,u(t,x),s(t,x)-1,q(t,x),q_x(t,x))\in X_m\left([0,\infty);\mathbb{I}\right).
		\end{eqnarray*}
		Moreover, there exist constants $C_0$ (independent of $t$) such that for any $t>0$,
		\begin{multline*}
		\vertiii{(P-1,u,s-1,q,q_x)(t)}^2_m+\int_{0}^{t}\vertiii{(DP,Ds)(s)}^2_{m-1}+\vertiii{(u,q,q_x)(s)}^2_m{\rm d}s\\[-2mm]\leq C_0\vertiii{\left(P-1,u,s-1\right)(0)}_m^2\ \
		\mathrel{\overset{\makebox[0pt]{\mbox{\scriptsize\sffamily \eqref{XmHm}}}}{\leq}}\ \
		C_0C_1\|(P_0-1,u_0,s_0-1)\|^2_{H^m}.
		\end{multline*}
		Here $DP=(P_t,P_x), Ds=(s_t,s_x)$ and the constant $C_1$ is defined in \eqref{XmHm}.
	\end{theorem}
	
	\begin{remark} Some remarks are listed below:
		\begin{itemize}
			\item In our main result, we ask the initial data is a small perturbation near the equilibrium state where $(P,u,s)=(1,0,1)$. For the corresponding global solvability results with large initial data, it is shown in \cite{MR4419016} for its one-dimensional Cauchy problem that its solution will develop singularity for a class of state equations and a class of large initial data even the thermal conductivity is taken into account, while the result obtained in \cite{MR4661717} shows that, for the initial-boundary value problem of the radiative Euler equations in a one-dimensional periodic box, if both viscosity and thermal conductivity are introduced, such an initial-boundary value problem does exists a unique global smooth solution for any large initial data.
			
			\item Although global smooth solutions are construceted in this paper, its large time behavior is not clear. Moreover, it would be interesting to consider the exterior problem, i.e., the case with $b=+\infty$. Such problems are under our current research.
			
		\end{itemize}
		
	\end{remark}

We can see it from the a priori estimates in Section \ref{a priori estimates} that the main difficulty to obtain the global solvability results from the high-dimensional space variable $r$, because the spatial derivative of $r$ \ep{referring to $\eqref{r_eq}_2$} is not small even with small initial data. Furthermore, we can see $r$ comes from the Lagrangian coordinate transformation, which gives us a beautiful structure of system \eqref{cns4}. However, $r$ as a function of  $\rho$ and $u$, its definition closely  related to the nonlinear structure of mass-conservation equation \eqref{mass-conservation}. This dependency will cause troubles when we try to obtain the local solution of linearized system. \ep{One can refer to the comments under \eqref{barrdefinition}, \eqref{barq} and \eqref{cns11}.}

As usual, the main proof of Theorem \ref{thm} can be divided into two parts, the a priori estimates and the local existence result, which are established respectively in Section \ref{a priori estimates}  and \ref{local}. If these two results stand, then with the method of continuity \ep{see for an example that \cite[p.~5972-5973]{MR3624545}}, Theorem \ref{thm} holds.

\section{A Priori Estimates}\label{a priori estimates}
 To simplify the a priori estimates, we reformulate our system around the equilibrium state by taking change of variables as that $(P,u,s,q)\rightarrow(P+1,u,s+1,q)$ and the system \eqref{cns4} is reformulated as
\begin{equation}
\label{cns5}
\left\{	
\begin{aligned}
&P_t+\frac{C_v+1}{C_v}(P+1)\rho(r^2u)_x+\frac{1}{C_v}\rho(r^2q)_x=0,\\[2mm]
&(r^2u)_t+r^4P_x=2ru^2,\\[2mm]
&s_t+\frac{(r^2q)_x}{\theta}=0,\\
&q-r^2\rho^2(r^2q)_{xx}+\frac{4r^2\theta^3}{C_v+1}\left(P_x+s_x\right)=\frac 1 2r^2(\rho^2)_x(r^2q)_x-\frac{4r^2\theta^3}{C_v+1}Ps_x,
\end{aligned}	
\right.
\end{equation}
as well as the linearized form that
\begin{equation}
\label{cns6}
\left\{	
\begin{aligned}
&P_t+\frac{C_v+1}{C_v}c_\rho(r^2u)_x+\frac{1}{C_v}(r^2q)_x=\mathbb{S}_1,\\
&u_t+r^2P_x=0,\\[2mm]
&s_t+\frac{1}{c_\theta}(r^2q)_x=\mathbb{S}_3,\\
&q-r^2c_\rho^2(r^2q)_{xx}+\frac{4r^2c_\theta^3}{C_v+1}\left(P_x+s_x\right)=\mathbb{S}_4,
\end{aligned}	
\right.
\end{equation}
where
\begin{eqnarray*}
  &&\mathbb{S}_1=\frac{C_v+1}{C_v}(c_\rho-\rho)(r^2u)_x+\frac{C_v+1}{C_v}\rho P(r^2u)_x+\frac{1}{C_v}(1-\rho)(r^2q)_x,\\
  &&\mathbb{S}_3=\left(\frac{1}{c_\theta}-\frac{1}{\theta}\right)(r^2q)_x,\\
  &&\mathbb{S}_4=r^2(\rho^2-c_\rho^2)(r^2q)_{xx}+\frac{4r^2}{C_v+1}(c_\theta^3-\theta^3)\left(P_x+s_x\right)+\frac 1 2r^2(\rho^2)_x(r^2q)_x-\frac{4r^2\theta^3}{C_v+1}Ps_x.
\end{eqnarray*}
The corresponding initial and boundary conditions are given by that
\begin{alignat}{2} \label{initial}
(P(0,x),u(0,x),s(0,x))&=
(P_0-1,u_0,s_0-1)(x),
\qquad && x\in\mathbb{I},\\[1mm]               \label{bdry}
(u(t,0),q(t,0))&=(u(t,1),q(t,1))=0,\qquad&& t\geq 0,
\end{alignat}
For later use, we list some Sobolev inequalities refined with respect to norm $\interleave \cdot \interleave_m$ as follows
\begin{lemma}[refined version of Appendix B in \cite{MR834481}]\label{inequality}
	Let $f,g\in X_{k}([0,\infty),\mathbb{I})$. Then
	\begin{itemize}
\begin{spacing}{1.5}
	\item[$(1)$]$\vertiii{fg}_1\lesssim\vertiii{f}_{1}\vertiii{g}_{1}$, with $k=1$; $\vertiii{fg}_k\lesssim \vertiii{f}_{k}\vertiii{g}_{k-1}+\vertiii{f}_{k}\vertiii{g}_{k-1}$,     with $k\geq2$; \ep{From now on, that $X\lesssim Y$ means $X\leq CY$ for some trivial constant $C$.}
\end{spacing}
	\vspace{0pt}
		\begin{spacing}{1.5}
				\item[$(2)$]$\|D^\alpha f D^\beta g\|\lesssim\vertiii{f}_{s_1}^{a_1}\vertiii{f}_{s_1-1}^{1-a_1}\vertiii{g}_{s_2}^{a_2}\vertiii{g}_{s_2-1}^{1-a_2}$, with $0<a_1,a_2<1$, provided $|\alpha|\leq s_1-1,\ |\beta|\leq s_2-1,\ |\alpha|+|\beta|+1\leq
				s_1+s_2,\ 1\leq s_1,s_2\leq k$;
		\end{spacing}
	\vspace{10pt}
		\item[$(3)$]$F(f,g)$ is a $C^k\ (k\geq2)$ function with respect to $(f,g)$, then
		\begin{align*}
			\vertiii{F}_k&\leq\sum_{|\gamma|=0}^k\|D^\gamma(F)\|
			\leq\sum_{|\gamma|=0}^k\sum\left\|\left(\partial^i_f\partial^j_gF\right)D^{\alpha^1}f\cdots D^{\alpha^i}fD^{\beta^1}g\cdots D^{\beta^j}g\right\|\\
			&\leq\left|F\right|_{L^\infty}C\{\mathbb{I}\}+\sum_{|\gamma|=1}^k\sum\left|\partial^i_f\partial^j_gF\right|_{L^\infty}
			\left\|D^{\gamma}\left(f^ig^j\right)\right\|\\
			&\lesssim\sum_{|\gamma|=0}^k\sum
			\left|\partial^i_f\partial^j_gF\right|_{L^\infty}
			\left(1+\vertiii{f}_{k-1}^{i-1}\vertiii{g}_{k-1}^{j}\vertiii{f}_{k}+\vertiii{f}_{k-1}^{i}\vertiii{g}_{k-1}^{j-1}\vertiii{g}_{k}\right)\\
			&\leq C\left\{\mathbb{I},\left\|F\right\|_{C^k}\right\}\left[\left(1+\vertiii{f}_{k-1}\right)^{k-1}\left(1+\vertiii{g}_{k-1}\right)^{k-1}\left(1+\vertiii{f}_{k}\right)\right.\\
			&\qquad\qquad\qquad\qquad\qquad\quad\left.+\left(1+\vertiii{f}_{k-1}\right)^{k-1}\left(1+\vertiii{g}_{k-1}\right)^{k-1}\left(1+\vertiii{g}_{k}\right)\right].
		\end{align*}
\begin{spacing}{1}
			Here $\alpha^1,\cdots,\alpha^i,\cdots,\beta^1,\cdots,\beta^j,\ \gamma$ are two-dimensional multi-index $($For an instance, Let $\alpha^i=\left(\alpha^i_1,\alpha^i_2\right)$, then $D^\alpha=\partial_t^{\alpha^i_1}\partial_x^{\alpha^i_2}$.$)$; $0\leq i+j \leq |\gamma|;\ |\alpha^1|+\cdots+|\alpha^i|+|\beta^1|+\cdots+|\beta^j|=|\gamma|$; $C\left\{\mathbb{I},\left\|F\right\|_{C^k}\right\}$ is a positive nondecreasing constant with respect to $\mathbb{I}$ and $\left\|F\right\|_{C^k}$.
\end{spacing}
	\end{itemize}
\end{lemma}

We define the solution spaces with parameters $T>0$ and $N>0$ that
\begin{align*}
&A_m\left([0,T];\mathbb{I};N\right)
=\left\{f(t,x)\ \Big|\ (t,x)\in[0,T]\times\mathbb{I},\ \vertiii{f}^2_{m,tan}+\vertiii{f}^2_{m-1}\leq N\right\},\\
&B_m\left([0,T];\mathbb{I};N\right)=\left\{f\in X_m\left([0,T];\mathbb{I}\right)\Big|\ \vertiii{f}_{m}\leq N\right\}.
\end{align*}
The a priori estimates $(P,u,s,q)$ in a certain solution space are derived in this subsection as follows.
\begin{proposition}[A priori estimate]\label{a priori}
	Suppose that the initial data $(P_0-1,u_0,s_0-1)$ satisfy the conditions in Theorem \ref{thm} and the initial value problem \eqref{cns2} has a unique solution $(P,u,s,q)$ satisfying $(P,u,s,q,q_x)\in B_m([0,T];\mathbb{I};\epsilon)$, where m is an integer satisfying $m\geq2$, $T$, $\epsilon$ are some positive constants and $\epsilon$ is small enough. Then there exists a constant $C_0>0$ , which is independent of $T$, such that
    \begin{multline}\label{apriori}
\vertiii{(P,u,s,q,q_x)(t)}^2_m+\int_{0}^{t}\vertiii{(DP,Ds)(s)}^2_{m-1}+\vertiii{(u,q,q_x)(s)}^2_m{\rm d}s\\[-2mm]
\leq C_0\vertiii{\left(P-1,u,s-1\right)(0)}_m^2\ \
\mathrel{\overset{\makebox[0pt]{\mbox{\scriptsize\sffamily \eqref{XmHm}}}}{\leq}}\ \
C_0C_1\|(P_0-1,u_0,s_0-1)\|^2_{H^m}:=V_0.
\end{multline}
\end{proposition}
The proof of Proposition \ref{a priori} is divided into the following three lemmas.
\begin{lemma}\label{m=0}
	Under the assumption of Proposition \ref{a priori}, for any $t\in[0,T]$, it holds that
	\begin{equation}\label{00}
		\frac{\rm d}{{\rm d}t}\|(P,u,s)(t)\|^2+\left\|\left[q,(r^2q)_x(t)\right]\right\|^2\lesssim\epsilon\left\|\left[P_x,s_x,(r^2u)_x\right](t)\right\|^2.
	\end{equation}
\end{lemma}
\begin{proof}
	By computing $\eqref{cns6}_1\times\frac{C_v}{(C_v+1)c_\rho}P+\eqref{cns6}_2\times u+\eqref{cns6}_3\times\frac{c_\theta s}{(C_v+1)c_\rho}+\eqref{cns6}_4\times\frac{1}{4c_\theta^3}q$ and integrating the resultant equality over $\mathbb{I}$, we have that
	\begin{align}\notag
		&\frac{{\rm d}}{{\rm d}t}\int_{\mathbb{I}}\left\{\frac{C_v}{(C_v+1)c_\rho}P^2+u^2+\frac{c_\theta }{(C_v+1)c_\rho}s^2\right\}\ {\rm d}x
		+\int_{\mathbb{I}}\frac{1}{4c_\theta^3}q^2+\frac{c_\rho^2}{4c_\theta^3}(r^2q)_x\ {\rm d}x\\\notag
		=&\int_{\mathbb{I}}\left[(\frac{\rho}{c_\rho}-1)P\right]_xr^2u-\left(\frac{\rho}{c_\rho} P^2\right)_xr^2u
		+\left[\frac{1}{C_v+1}\left(\frac{\rho}{c_\rho}-1\right)P\right]_xr^2q
		+\left[\left(\frac{c_\theta}{\theta}-1\right)\frac{s}{C_v+1}\right]_xr^2q+\\\notag
		&\left[\frac{r^2}{4c_\theta^3}(c_\rho^2-\rho^2)q\right]_x(r^2q)_{x}
		+\frac{r^2(c_\theta^3-\theta^3)}{(C_v+1)c_\theta^3}q\left(P_x+s_x\right)
		+\frac{1}{8c_\theta^3}r^2(\rho^2)_xq(r^2q)_x
		-\frac{r^2\theta^3}{(C_v+1)c_\theta^3}Pqs_x \ {\rm d}x	\\\notag
		\lesssim&\ \epsilon\left\|\left[P_x,s_x,r^2u,q,(r^2q)_x\right](t)\right\|^2	
	\end{align}
	With $Poincar\acute{e}$ inequality, \eqref{00} stands.
\end{proof}

\begin{lemma}\label{1jie}
	Under the assumption of Proposition \ref{a priori}, for any $t\in[0,T]$, it holds that
	\begin{equation*}
	  \vertiii{(P,u,s,q,q_x)(t)}_1^2+\int_{0}^{t}\left\|\left(DP,Ds\right)(s)\right\|^2+\vertiii{\left(u,q,q_x\right)(s)}_1^2{\rm d}s\leq V_0.
	\end{equation*}
\end{lemma}
\begin{proof}
By computing
\begin{multline}\label{Pusqx}
	\partial_x\eqref{cns5}_1\times\frac{C_v}{C_v+1}r^4P_x+\partial_x\eqref{cns5}_2\times \rho(P+1)(r^2u)_x\\+\partial_x\eqref{cns5}_3\times\frac{1}{C_v+1}(P+1)r^4s_x+\eqref{cns5}_4\times\frac{1}{(C_v+1)r^2\rho}(r^4P_x+r^4s_x)
\end{multline}
and integrating the resultant equality over $\mathbb{I}$, we have that
 \begin{multline}\label{10}
 	\frac{\rm d}{{\rm d}t}\int_{\mathbb{I}}\frac{C_v}{2(C_v+1)}r^4P_x^2+\frac{1}{2}\rho(P+1)\left[(r^2u)_x\right]^2+\frac{1}{2(C_v+1)}(P+1)r^4s_x^2\ {\rm d}x+\\
 	\int_{0}^{t}\int_{\mathbb{I}}\frac{4\theta^3}{\rho(C_v+1)^2}(r^2P_x+r^2s_x)^2{\rm d}x{\rm d}s\lesssim\epsilon\left\|\left[P_x,s_x,(r^2u)_x\right](t)\right\|^2+\left\|\left[q,(r^2q)_x\right](t)\right\|^2.
 \end{multline}
 Here that $\int_{\mathbb{I}}\partial_x\left[\left(r^4P_x\right)(r^2u)_x\right]{\rm d}x=0$ is because $r^4P_x=2ru^2-(r^2u)_t=0$ on $\partial\mathbb{I}$. To make up for the dissipation of $(r^2u)_x$, we compute $\left(\eqref{cns6}_1+\eqref{cns6}_2\right)\times(r^2u)_x$
  and integrate the resultant equality over $\mathbb{I}$ to obtain
  \begin{equation}\label{11}
  	\frac{\rm d}{{\rm d}t}\int_{\mathbb{I}}(P+s)(r^2u)_x\ {\rm d}x+\left\|(r^2u)_x\right\|^2
  	\lesssim \left\|(r^2q)_x\right\|^2+\left\|P_x+s_x\right\|^2+\epsilon\left(\|P_x\|^2+\|(r^2u)_x\|^2\right).
  \end{equation}
As for the dissipation of $P_x$, we multiply $\eqref{cns5}_2$ by $\times P_x$ and integrate the resultant equality over $\mathbb{I}$ to get
  \begin{equation}\label{12}
  	\frac{\rm d}{{\rm d}t}\int_{\mathbb{I}}r^2uP_x\ {\rm d}x+\left\|P_x\right\|^2\lesssim \left\|\left[(r^2u)_x,(r^2q)_x\right]\right\|^2.
  \end{equation}
Finally take $\eqref{00}+\delta_3\eqref{10}+\delta_2\eqref{11}+\delta_1\eqref{12}$ and we have
\begin{align*}
	\frac{\rm d}{{\rm d}t}\int_{\mathbb{I}}&(P^2+u^2+s^2)+
	\delta_3\left\{\frac{C_v}{2(C_v+1)}r^4P_x^2+\frac{1}{2}\rho(P+1)\left[(r^2u)_x\right]^2+\frac{1}{2(C_v+1)}(P+1)r^4s_x^2\right\}\\[2mm]
	&\qquad\qquad\quad\quad
	+\delta_2\left[(P+s)(r^2u)_x\right]+\delta_1\left(r^2uP_x\right){\rm d}x\\[2mm]
	&\quad\qquad\qquad\qquad\quad\quad+\left\{\left\|\left[q,(r^2q)_x\right]\right\|^2+\delta_3\|P_x+s_x\|^2+\delta_2\left\|(r^2u)_x\right\|^2+\delta_1\|P_x\|^2\right\}\leq0
\end{align*}
where $\delta_i\ (i=1,2,3)$ are positive constants satisfying $\epsilon\ll\delta_1\ll\delta_2\ll\delta_3\ll1$. Integrate the last inequality over $(0,t)$, we have that
\begin{align}\label{pus1x}
	\left\|(P,u,s)(t)\right\|_{H^1}^2+\int_{0}^{t}\left\|\left(P_x,s_x\right)(s)\right\|^2+\left\|\left(u,q\right)(s)\right\|_{H^1}^2{\rm d}s\leq V_0.
\end{align}
Here the estimate of $\int_{0}^t\|u(s)\|^2{\rm d}s$ is obtained by $Poincar\acute{e}$ inequality, since $r^2u=0$ on $\partial\mathbb{I}$.
Let $\eqref{cns6}_4\times q$ and integrate the resultant equality over $\mathbb{I}$ to obtain
\begin{equation}\label{q1}
\|q(t)\|^2+\|(r^2q)_x(t)\|^2\lesssim\|P_x(t)\|^2+\|s_x(t)\|^2.
\end{equation}
Then we take $\eqref{cns5}_4\times (r^2q)_{xx}$ and integrate the resultant equality over $\mathbb{I}$ to get
\begin{equation}\label{q2}
\|(r^2q)_{xx}(t)\|^2\lesssim\|P_x(t)\|^2+\|s_x(t)\|^2.
\end{equation}
With \eqref{q1}, \eqref{q2} and \eqref{pus1x}, we have
that
\begin{align}\label{pusq1x}
	\left\|(P,u,s,q,q_x)(t)\right\|_{H^1}^2+\int_{0}^{t}\left\|\left(P_x,s_x\right)(s)\right\|^2+\left\|\left(u,q,q_x\right)(s)\right\|_{H^1}^2{\rm d}s\leq V_0.
\end{align}
From the first three equations in \eqref{cns6}, we can see that $\left\|(P_t,u_t,s_t)\right\|$ are bounded by the left side of \eqref{pusq1x}, which gives
\begin{align}\label{pus1x1t}
\vertiii{(P,u,s)(t)}_1^2+\left\|(q,q_x)(t)\right\|_{H^1}^2+\int_{0}^{t}\left\|\left(DP,Ds\right)(s)\right\|^2
+\vertiii{u(s)}_1^2
+\left\|(q,q_x)(s)\right\|_{H^1}^2{\rm d}s\leq V_0
\end{align}
	By taking $\partial_t\eqref{cns5}_4\times q_t$ and integrating the resultant equality over $\mathbb{I}$, we have that
	\begin{align}\notag
	&\|q_t(t)\|^2+\|(r^2q)_{xt}(t)\|^2\lesssim\|\left(P_t+s_t\right)\|\|(r^2q)_{xt}(t)\|^2+\epsilon\left(\|q_t\|^2+\|(r^2q)_{xx}\|^2\right)+\epsilon\|(r^2q)_{xt}\|^2\\[2mm]\notag
	&+\|(P_x,s_x)(t)\|^2+\int_{\mathbb{I}}s_t\left\{\frac{4\theta^3}{C_v+1}P\left[\partial_t(r^2q)-2ruq\right]\right\}_x-\rho_t\Big\{r^2\rho(r^2q)_x\left[\partial_t(r^2q)-2ruq\right]\Big\}_x{\rm d}x\\[2mm]\label{qt1}
	&\lesssim\epsilon\left(\|q_t\|^2+\|(r^2q)_{xt}\|^2\right)+\|(r^2q)_{xx}\|^2+\|(P_x,u_x,s_x,s_t,q)\|^2.
	\end{align}
By combining \eqref{qt1} and \eqref{pus1x1t}, we complete the proof of Lemma \ref{1jie}.
\end{proof}
\begin{remark}
	In \eqref{10}, we have utilized $\left|\left(P_x,P_t\right)(t)\right|_{L^\infty}\lesssim\left\|\left(P_x,P_t\right)(t)\right\|_{H^1}\lesssim\epsilon$, which is why we require $m\geq2$.
\end{remark}
\begin{remark}
	From \eqref{q1}, it can be easily derived that
	\begin{align*}
		\left\|q_x(t)\right\|^2\lesssim\left\|(r^2q_x)(t)\right\|^2\lesssim\|q(t)\|^2+\|(r^2q)_x(t)\|^2\lesssim\|P_x(t)\|^2+\|s_x(t)\|^2\leq V_0.
	\end{align*}
	And we can see if $\vertiii{q(t)}_k\leq V_0$\ep{$k\geq0$}, that  $\vertiii{(r^2q)_x(t)}_k\leq V_0$ is equivalent to that $\vertiii{q_x}_k\leq V_0$. And the similar conclusion also stands for $u$. These will be applied directly for the rest of this paper.
\end{remark}
 In the following Lemma, we use the induction method to prove the a priori estimates when $m\geq2$
\begin{lemma}\label{m>2}
	Under the assumption of Proposition \ref{a priori}, for any $t\in[0,T]$, it holds that
	\begin{equation}\label{m}
	\vertiii{\left(P,u,s,q,q_x\right)(t)}_m^2+\int_{0}^{t}\vertiii{(DP,Ds)(s)}^2_{m-1}+\vertiii{(u,q,q_x)(s)}^2_m{\rm d}s\leq V_0.
	\end{equation}
\begin{proof}
We assume \eqref{m} holds for any integer $m=k(k\geq2)$. Since \eqref{m} has been proved for $m=1$, we aim to prove \eqref{m} for $m=k+1$ and what's left is to obtain the estimates for $k+1$ order derivatives of $\left(P,u,s,q,q_x\right)$.

First, we derive the estimates of $P$ and $u$. Similar to what we did in Lemma \ref{1jie}, compute that
\begin{multline}\label{Pusqtkx}
\partial_t^k\partial_x\eqref{cns5}_1\times\frac{C_v}{C_v+1}\partial_t^k\left(r^4P_x\right)
+\partial_t^k\partial_x\eqref{cns5}_2\times\partial_t^k\left[\rho(P+1)(r^2u)_x\right]+
\partial_t^k\partial_x\eqref{cns5}_3\times\\\frac{1}{C_v+1}(P+1)\partial_t^k\left(r^4s_x\right)
+\partial_t^k\left[r^2\cdot\eqref{cns5}_4\right]\times\frac{1}{(C_v+1)r^4\rho}\left[\partial_t^k\left(r^4P_x\right)
+\partial_t^k\left(r^4s_x\right)\right]
\end{multline}
and integrate the resultant equality over $\mathbb{I}$ to get that
\begin{align}\notag
  &\frac{{\rm d}}{{\rm d}t}\int_{\mathbb{I}}\left\{\frac{C_v}{2(C_v+1)}r^4(\partial_t^k\partial_xP)^2
  +\frac{C_v}{C_v+1}\partial_t^k\partial_xP\left[\partial_t^k,r^4\right]P_x
  +\frac{1}{2}\rho(P+1)\left[\partial_t^k\partial_x(r^2u)\right]^2\right.\\\notag
  &\qquad+\partial_t^k\partial_x(r^2u)\left[\partial_t^k,\rho(P+1)\right]_t(r^2u)_x\notag
  +\frac{1}{2(C_v+1)}(p+1)r^4\left(\partial_t^k\partial_xs\right)^2\\\notag
  &\qquad+\left.\frac{1}{C_v+1}\partial_t^k\partial_xs(p+1)\left[\partial_t^k,r^4\right]s_x\right\}{\rm d}x
  +\left\|\partial_t^k\left(r^4P_x\right)+\partial_t^k\left(r^4s_x\right)(t)\right\|^2\\\notag
  &\lesssim\int_{\mathbb{I}}\frac{C_v}{2(C_v+1)}4r^3u\left(\partial_t^k\partial_xP\right)^2
  +\frac{C_v}{C_v+1}\partial_t^k\partial_xP\frac{{\rm d}}{{\rm d}t}\left\{\left[\partial_t^k,r^4\right]P_x\right\}
  \\\notag
  &\qquad-\frac{1}{C_v+1}\left[\partial_t^k\partial_x,\rho\right](r^2q)_x\partial_t^k(r^4P_x)
  \ {\rm d}x+\int_{\mathbb{I}}\frac{1}{2}\frac{{\rm d}}{{\rm d}t}\left[\rho(P+1)\right]\left[\partial_t^k\partial_x(r^2u)\right]^2\\\notag
  &\qquad+\underbrace{\partial_t^k\partial_x(r^2u)\frac{{\rm d}}{{\rm d}t}\left\{\left[\partial_t^k,\rho(P+1)\right](r^2u)_x\right\}}_{\circled{1}}
  +\underbrace{\partial_t^k\partial_x(2ru^2)\partial_t^k\left[\rho(P+1)(r^2u)_x\right]}_{\circled{2}}\ {\rm d}x\\\notag
  &\qquad+\int_{\mathbb{I}}\frac{1}{2(C_v+1)}\frac{{\rm d}}{{\rm d}t}\left[(P+1)r^4\right]\left(\partial_t^k\partial_xs\right)^2
  +\frac{1}{C_v+1}\partial_t^k\partial_xs\frac{{\rm d}}{{\rm d}t}\left\{(P+1)\left[\partial_t^k,r^4\right]s_x\right\}\\\notag
  &\qquad-\frac{1}{C_v+1}(P+1)\left[\partial_t^k\partial_x,\frac{1}{\theta}\right](r^2q)_x\ \partial_t^k\left(r^4s_x\right)\ {\rm d}x
  +\\\notag
  &\qquad\int_{\mathbb{I}}-\frac{1}{r^4(C_v+1)r^4\rho}\left[\partial_t^k,r^4\rho^2\right](r^2q)_{xx} \left[\partial_t^k\left(r^4P_x\right)+\partial_t^k\left(r^4s_x\right)\right]\\\notag
  &\qquad-\frac{1}{r^4(C_v+1)r^4\rho}\left[\partial_t^k,\frac{4\theta^3}{C_v+1}\left(r^4P_x+r^4s_x\right)\right]\left[\partial_t^k\left(r^4P_x\right)+\partial_t^k\left(r^4s_x\right)\right]\\\notag
  &\qquad-\frac{1}{r^4(C_v+1)r^4\rho}\partial_t^k(r^2q)\left[\partial_t^k\left(r^4P_x\right)+\partial_t^k\left(r^4s_x\right)\right]\\\notag
  &\qquad+\partial_t^k\left[r^4\rho\rho_x(r^2q)_x\right]\frac{1}{r^4(C_v+1)r^4\rho}\left[\partial_t^k\left(r^4P_x\right)+\partial_t^k\left(r^4s_x\right)\right]\\\notag
  &\qquad-\partial_t^k\left(\frac{4\theta^3r^4}{C_v+1}\right)\frac{1}{r^4(C_v+1)r^4\rho}\left[\partial_t^k\left(r^4P_x\right)+\partial_t^k\left(r^4s_x\right)\right]\ {\rm d}x\\\label{tkx1}
  &\lesssim\epsilon\left\|\partial_t^k\partial_x(P,s,r^2u)\right\|^2+ \vertiii{(P,u,s,q,q_x)} _k^2.
\end{align}
Here $[\cdot,\cdot]$ denotes the commutator of operators. As for the last inequality, we use the following two examples to illustrate how the estimates are done.
\begin{align}
	\int_{\mathbb{I}}\circled{1}\ {\rm d}x
	\lesssim\epsilon\left\|\partial_t^k\partial_x(r^2u)\right\|^2
	+\left\|\frac{{\rm d}}{{\rm d}t}\left\{\left[\partial_t^k,\rho(P+1)\right](r^2u)_x\right\}\right\|^2
\end{align}
Here,
\begin{align*}
	\left\|\frac{{\rm d}}{{\rm d}t}\left\{\left[\partial_t^k,\rho(P+1)\right](r^2u)_x\right\}\right\|^2
	=&\left\|\sum_{i=1}^{k}\partial_t^{i+1}\left[\rho\left(P+1\right)\right]\partial_t^{k-i}\left(r^2u\right)_x+\partial_t^i\left[\rho\left(P+1\right)\right]\partial_t^{k+1-i}\left(r^2u\right)_x\right\|^2\\
	\lesssim&\sum_{i=0}^{k}\left\|\partial_t^{i+1}\left[\rho\left(P+1\right)\right]\partial_t^{k-i}\left(r^2u\right)_x\right\|^2
\end{align*}
When $i=0$,
\begin{align}\notag
	\left\|\partial_t\left[\rho\left(P+1\right)\right]\partial_t^{k}\left(r^2u\right)_x\right\|^2	\lesssim\left|\partial_t\left[\rho\left(P+1\right)\right]\right|^2_{L^\infty}\left\|\partial_t^k\partial_x(r^2u)\right\|^2\lesssim\epsilon\left\|\partial_t^k\partial_x(r^2u)\right\|^2.
\end{align}
when $i=k$, $\partial_t^{k+1}P$ and $\partial_t^{k+1}s$ would appear. To cope with that, we apply  $\partial_t^k$ on the first and third equation in \eqref{cns5} which enable us to substitute $\partial_t^{k+1}P$ and $\partial_t^{k+1}s$ with terms containing $\partial_t^{k}\partial_x(r^2u)$ and $\partial_t^{k}\partial_x(r^2q)$. So we have
\begin{align*}
	\left\|\partial_t^{k+1}\left[\rho\left(P+1\right)\right]\left(r^2u\right)_x\right\|^2
	\lesssim&\left|(r^2u)_x\right|^2_{L^\infty}\left(\left\|\partial_t^{k}\partial_x(r^2u)\right\|^2+\left\|\partial_t^{k}\partial_x(r^2q)\right\|^2
	+\epsilon\vertiii{(P,s)}_k^2\right)\\
	\lesssim&\epsilon\left(\left\|\partial_t^{k}\partial_x(r^2u)\right\|^2+\left\|\partial_t^{k}\partial_x(r^2q)\right\|^2\right)+\vertiii{(P,s)}_k^2
\end{align*}
When $i=k-1$,
\begin{align*}
	\left\|\partial_t^{k}\left[\rho\left(P+1\right)\right]\partial_t\left(r^2u\right)_x\right\|^2
	\lesssim&\left\|\partial_t\partial_x(r^2u)\right\|^2\left|\partial_t^{k}\left[\rho\left(P+1\right)\right]\right|^2_{L^\infty}\lesssim\epsilon\left\|\partial_t^{k}\left(P,s\right)\right\|_{H^1}^2.
\end{align*}
Similarly the term can be estimated when $1\leq i\leq k-2$.
As for $\circled{2}$, we have
\begin{align*}
		\int_{\mathbb{I}}\circled{2}\ {\rm d}x\lesssim&\epsilon\left\|\partial_t^k\left[\rho(P+1)(r^2u)_x\right]\right\|^2
		+\left\|\partial_t^k\partial_x(2ru^2)\right\|^2\\
		\lesssim&\epsilon\left(\left\|\partial_t^k\partial_x(r^2u)\right\|^2+\vertiii{(P,s)}_k^2\vertiii{(r^2u)}_k^2\right)+\epsilon\left(\left\|\partial_t^k\partial_xu\right\|^2+\vertiii{u}_k^2\right).
\end{align*}
In the last inequality, we have utilized a more detailed version of Lemma \ref{inequality} $(1)$ to obtain the estimate that
\begin{align*}
	&\left\|\partial_t^k\partial_x(2ru^2)\right\|\lesssim\vertiii{r}_{k}\vertiii{u}_{k}\left(\left\|\partial_t^k\partial_xu\right\|+\vertiii{u}_{k}\right)+\vertiii{r}_{k+1}\vertiii{u}_{k}\vertiii{u}_{k}.
\end{align*}
To make up for the dissipation of $\partial_t^k\partial_x(r^2u)$, we compute $\partial_t^k\left(\eqref{cns5}_1+\eqref{cns5}_3\right)\times\partial_t^k\partial_x(r^2u)$ and integrate the resultant equality over $\mathbb{I}$ to get that
\begin{align}\notag
    \frac{{\rm d}}{{\rm d}t}&\int_{\mathbb{I}}\left[\left(\partial_t^kP+\partial_t^ks\right)\partial_t^k\partial_x(r^2u)\right]{\rm d}x
	+\int_{\mathbb{I}}\frac{C_v+1}{C_v}\rho(P+1)\left[\partial_t^k\partial_x(r^2u)\right]^2
	-\left[\left(\partial_t^kP+\partial_t^ks\right)\partial_t^k(r^2u)\right]_x {\rm d}x\\\notag
	=&\int_{\mathbb{I}}-\left(\partial_t^kP+\partial_t^ks\right)\partial_t^{k+1}(r^2u)
	-\frac{C_v+1}{C_v}\left[\partial_t^k,\rho(P+1)\right](r^2u)_x\partial_t^k\partial_x(r^2u)\\\label{tk1}
	&\quad -\frac{1}{C_v}\partial_t^k\partial_x(r^2q)\partial_t^k\partial_x(r^2u)
	-\partial_t^k\left[\frac{(r^2q)_x}{\theta}\right]\partial_t^k\partial_x(r^2u)\ {\rm d}x.
\end{align}
Apply $\partial_t^k$ on the third equation in \eqref{cns5} and we have
\begin{align}\label{tku}
	\partial_t^{k+1}(r^2u)=-\partial_t^k(r^4P_x)+\partial_t^k(2ru^2).
\end{align}
Substitute \eqref{tku} into the first term on the right side of \eqref{tk1} to get
\begin{multline}\label{tkx1u}
   \|\partial_t^k(r^2u)_x\|^2+\frac{{\rm d}}{{\rm d}t}\int_{\mathbb{I}}\left(\partial_t^kP+\partial_t^ks\right)\partial_t^k\partial_x(r^2u)\ {\rm d}x\\\lesssim\epsilon\left\|\partial_t^k(r^4P_x)\right\|^2+\left\{\left\|\partial_t^k\partial_xP+\partial_t^k\partial_xs\right\|^2+ \vertiii{\left[P,u,s,(r^2q)_x\right]}_k^2\right\}.    	
\end{multline}
Similarly we add the dissipation of $\partial_t^k(r^4P_x)$ by having $\eqref{tku}\times\partial_t^k(r^4P_x)$ and integrating it over $\mathbb{I}$, which gives
\begin{align}\label{tkx1P}
	\|\partial_t^k(r^4P_x)\|^2+\frac{{\rm d}}{{\rm d}t}\int_{\mathbb{I}}\left[\partial_t^k(r^2u)\partial_t^k(r^4P_x)\right]\ {\rm d}x\lesssim\vertiii{\left[P,u,s,(r^2q)_x\right]}_k^2+\left\|\partial_t^k\partial_x(r^2u)\right\|^2.
\end{align}
By taking $\eqref{tkx1}+\delta_7\eqref{tkx1u}+\delta_6\eqref{tkx1P}$ ($\epsilon\ll\delta_6\ll\delta_7\ll1$) and integrating the resultant inequality over $(0,t)$, we have that
\begin{align}\label{sum2}
\|\partial_t^k\partial_x\left(P,u,s\right)(t)\|^2+\int_{0}^{t}\|\partial_t^k\partial_x\left(P,u,s\right)(t)\|^2{\rm d}s\leq V_0.
\end{align}
With \eqref{sum2}, if we apply $\partial_t^{k-i}\partial_x^i$ on the first two equation in \eqref{cns5}, iteratively we can have the estimates that
\begin{align}
	\|\partial_t^{k-i}\partial_x^{i+1}\left(P,u\right)(t)\|^2+\int_{0}^{t}\|\partial_t^{k+1-i}\partial_x^i\left(P,u\right)(t)\|^2{\rm d}s\leq V_0,\quad\text{for }i=1,\cdots,k.
\end{align}
Instead of $\partial_t^{k-i}\partial_x^i$, if we apply $\partial_t^{k}$ and we can get the estimate for $\partial_t^{k+1}\left(P,u\right)$.

Second, we derive the estimates of $s$. Since we already have the estimates of $\vertiii{(q,q_x)}_k$, which together with the structure of third equation in \eqref{cns5} enable us to estimate all the $k+1$ order derivatives of $s$ except of $\partial_x^{k+1}s$. To overcome that, we compute $\partial_x^{k+1}\eqref{cns5}_3\times\partial_x^{k+1}s$ and integrate the resultant equality over $\mathbb{I}$ to get
\begin{align*}
	\frac 12\frac{{\rm d}}{{\rm d}t} \left\|\partial_x^{k+1}s(t)\right\|^2+
	\int_{\mathbb{I}}\frac{\partial_x^{k+2}\left(r^2q\right)}{\theta}\ \partial_x^{k+1}s\ {\rm d}x
	=\int_{\mathbb{I}}\left[\partial_x^{k+1},\frac{1}{\theta}\right]\left(r^2q\right)_x\partial_x^{k+1}s{\rm d}x.
\end{align*}
Then we compute $\partial_x^k\left(\frac{1}{r^2\rho^2}\eqref{cns5}_4\right)$, which can be used to represent $\partial_x^{k+2}\left(r^2q\right)$ in the last equality. After substitute for $\partial_x^{k+2}\left(r^2q\right)$, we have
\begin{align}\notag
	&\frac{{\rm d}}{{\rm d}t}\left\|\partial_x^{k+1}s(t)\right\|^2
	+\int_{\mathbb{I}}\frac{8\theta^2}{\left(C_v+1\right)\rho^2}\left(\partial_x^{k+1}s\right)^2{\rm d}x
	\lesssim\epsilon\left\|\partial_x^{k+1}s\right\|^2
	+\left\|\left[\partial_x^k,\frac{4\theta^3}{\left(C_v+1\right)\rho^2}\right]s_x\right\|^2
	+\\\notag
	&\left\|\partial_x^k\left(\frac{q}{r^2\rho^2}+\frac{4\theta^3}{\left(C_v+1\right)\rho^2}P_x+\frac{\rho_x}{\rho}(r^2q)_x\right)\right\|^2
	+\left\|\left[\partial_x^k,\frac{4\theta^3}{\left(C_v+1\right)\rho^2}P\right]s_x\right\|^2
	+\left\|\left[\partial_x^{k+1},\frac{1}{\theta}\right]\left(r^2q\right)_x\right\|^2\\[2mm]\label{xk+1s}
	&\qquad\qquad\qquad\qquad\qquad\qquad\qquad\qquad\qquad\quad\  \lesssim\epsilon\left\|\partial_x^{k+1}s\right\|^2+\vertiii{\left(s,q,q_x\right)}_{k}^2+\vertiii{P}_{k+1}^2.
\end{align}
Integrate the inequality above over $[0,t]$ and we have
\begin{align}
	\left\|\partial_x^{k+1}s(t)\right\|^2
	+\int_{0}^t\left\|\partial_x^{k+1}s(\tau)\right\|^2{\rm d}\tau\leq V_0.
\end{align}
Now we use the following two estimates to illustrate why \eqref{xk+1s} stands.
\begin{align*}
    \left\|\partial_x^k\left(\frac{q}{r^2\rho^2}\right)\right\|\quad
	&\mathrel{\overset{\makebox[0pt]{\mbox{\tiny\sffamily Lemma \ref{inequality}\ep{1}}}}{\leq}}\quad
	\vertiii{\frac{1}{r^2\rho^2}}_k\vertiii{q}_k\\
		&\mathrel{\overset{\makebox[0pt]{\mbox{\tiny\sffamily Lemma \ref{inequality}\ep{3}}}}{\leq}}\quad
		\left\|\frac{1}{r^2\rho^2}\right\|_{C_k(r,\rho)}\left[\left(1+\vertiii{r}_k\right)^{k-1}\vertiii{\rho}_k+\left(1+\vertiii{\rho}_k\right)^{k-1}\vertiii{r}_k\right]\vertiii{q}_k\lesssim \vertiii{q}_k,
		\end{align*}
		\begin{align*}
	&\left\|\left[\partial_x^{k+1},\frac{1}{\theta}\right]\left(r^2q\right)_x\right\|=\left\|D^\alpha\left(\frac{1}{\theta}\right)D^\beta\left(r^2q\right)_x\right\|\\[3mm]
	\ \ &\qquad\lesssim\left\{
	\begin{aligned}
	&\left|D\left(\frac{1}{\theta}\right)\right|_{L^{\infty}}\vertiii{\left(r^2q\right)_x}_{k}
	\lesssim\vertiii{\frac{1}{\theta}}_2\vertiii{\left(r^2q\right)_x}_{k},\ \begin{aligned}
	&|\alpha|=1\\&|\beta|=k
	\end{aligned}\\[5mm]
	&\vertiii{\frac{1}{\theta}}_{k+1}\left|\left(r^2q\right)_x\right|_{L^{\infty}}
	\lesssim\vertiii{\frac{1}{\theta}}_{k+1}\vertiii{\left(r^2q\right)_x}_1,\qquad\ \begin{aligned}
	&|\alpha|=k+1\\&|\beta|=0
	\end{aligned}\\[5mm]
	&\vertiii{\frac{1}{\theta}}_{k+1}\vertiii{\left(r^2q\right)_x}_{k}\ (\text{with Lemma \ref{inequality} (2)}),\quad\begin{aligned}
	&|\alpha|+|\beta|<k+1\\
	&|\alpha|\leq k\\
	&|\beta|\leq k-1
	\end{aligned}
	\end{aligned}
	\right.	\\
	&\qquad\lesssim\vertiii{\frac{1}{\theta}}_{k+1}\vertiii{\left(r^2q\right)_x}_{k}\lesssim\left(1+\vertiii{(P,s)}_{k+1}\right)\vertiii{\left(r^2q\right)_x}_{k}\\
	&\qquad\lesssim\vertiii{\left(r^2q\right)_x}_{k}+\vertiii{P}_{k+1}+\vertiii{s}_{k}+\sum_{i=1}^{k+1}\left\|\partial_x^{k+1-i}\partial_t^is\right\|+\epsilon\left\|\partial_x^{k+1}s\right\|.
\end{align*}
Third, Let's go for the estimates of $q$. According to the induction hypothesis,
\begin{align*}
	\vertiii{(q,q_x)(t)}_k^2+\int_{0}^t\vertiii{(q,q_x)(\tau)}_k^2\ {\rm d}\tau\leq V_0.
\end{align*}
We aim to prove above estimates for $\vertiii{(q,q_x)}_{k+1}$. The remaining terms in $\vertiii{(q,q_x)}_{k+1}$ are $\left\|\partial_t^{k+1}q\right\|$ and $\left\|D^\alpha q_x\right\|$ $(\,|\alpha|=k+1\,)$. Apply $\partial_t^{k-i}\partial_x^i$ $(0\leq i\leq k)$ on the fourth equation in \eqref{cns5} and we have
\begin{align*}
\left\|\partial_t^{k-i}\partial_x^{i+2}(r^2q)\right\|^2
\lesssim\vertiii{\left(DP,Ds\right)}^2_{k}+\vertiii{u}_{k+1}+\vertiii{(q,q_x)}_{k}^2.
\end{align*}
The estimates that are left to derive are those of $\left\|\partial_t^{k+1}q\right\|$ and $\left\|\partial_t^{k+1}\partial_xq\right\|$.
We apply $\partial_t^{k+1}$ on the fourth equation in \eqref{cns5}, multiply the resultant equality by $\partial_t^{k+1}q$ and integrate  it over $\mathbb{I}$, which gives that
\begin{align}\notag
&\left\|\partial_t^{k+1}q\right\|^2+\int_{\mathbb{I}}\rho^2\left[\partial_t^{k+1}\partial_x\left(r^2q\right)\right]^2{\rm d}x
=\\\notag
&\qquad\qquad\quad\int_{\mathbb{I}}-2\rho\rho_x\partial_t^{k+1}\partial_x\left(r^2q\right)\partial_t^{k+1}\left(r^2q\right)
+\partial_t^{k+1}\partial_x\left(r^2q\right)\left[\left(\left[\partial_t^{k+1},r^2\right]q\right)\rho^2\right]_x\\\notag
&\qquad\qquad\quad+\Big\{\partial_t^{k+1}\left[r^2\rho^2(r^2q)_{xx}\right]-r^2\rho^2\partial_t^{k+1}(r^2q)_{xx}\Big\}\partial_t^{k+1}q\\[1.5mm]\notag
&\qquad\qquad\quad-\left\{\partial_t^{k+1}\left[\frac{4r^2\theta^3}{C_v+1}\left(P_x+s_x\right)\right]
-\frac{4r^2\theta^3}{C_v+1}\partial_t^{k+1}\partial_x  \left(P+s\right)\right\}\partial_t^{k+1}q\\\notag
&\qquad\qquad\quad+\partial_t^{k+1}\left(P+s\right)\left\{\left[\partial_t^{k+1}(r^2q)-\left(\left[\partial_t^{k+1},r^2\right]q\right)\right]\frac{4\theta^3}{C_v+1}\right\}_x\\\notag
&\qquad\qquad\quad+\Big\{\partial_t^{k+1}\left[r^2\rho\rho_x(r^2q)_x\right]-r^2\rho\partial_t^{k+1}\rho_x(r^2q)-r^2\rho\rho_x\partial_t^{k+1}(r^2q)_x\Big\}\partial_t^{k+1}q\\\notag
&\qquad\qquad\quad-\partial_t^{k+1}\rho\Big\{\rho(r^2q)_x\left[\partial_t^{k+1}(r^2q)-\left[\partial_t^{k+1},r^2\right]q\right]\Big\}_x
+r^2\rho\rho_x\partial_t^{k+1}(r^2q)_x\partial_t^{k+1}q\\\notag
&\qquad\qquad\quad-\left[\partial_t^{k+1}\left(\frac{4r^2\theta^3}{C_v+1}Ps_x\right)-\frac{4r^2\theta^3}{C_v+1}P\partial_t^{k+1}s_x\right]\partial_t^{k+1}q\\\notag
&\qquad\qquad\quad+\partial_t^{k+1}s\left\{\frac{4\theta^3}{C_v+1}P\left[\partial_t^{k+1}(r^2q)-\left[\partial_t^{k+1},r^2\right]q\right]\right\}_x\\\notag
&\quad\quad
\lesssim\epsilon\left(\left\|\partial_t^{k+1}q\right\|^2+\left\|\partial_t^{k+1}\partial_x\left(r^2q\right)\right\|^2\right)+\left\|\partial_t^{k}\partial_x^2(r^2q) \right\|^2+\vertiii{\left(DP,Ds\right)}^2_{k}+\vertiii{u}_{k+1}+\vertiii{(q,q_x)}_{k}^2.
\end{align}
Now we have completed our proof of Lemma \ref{m>2}.
\end{proof}	
\end{lemma}
\begin{remark}
In the Lemma \ref{m>2}, we prove \eqref{m} when $m\geq2$. In this more general case, we can see that the essence of proof is to obtain the estimates for $(\partial_t^{m-1}\partial_xP,\partial_t^{m-1}\partial_xu,\partial_t^{m-1}\partial_xs)$.
\end{remark}
\begin{remark}
	The reason why we use the complicated multipliers in \eqref{Pusqx} and \eqref{Pusqtkx} such as $\partial_t^k\left[\rho\left(P+1\right)\left(r^2u\right)_x\right]$ rather than simply $\partial_t^k\partial_xu$ is to overcome the difficulty that $r_x=\frac{1}{r^2\rho}$ is not small. Otherwise the terms like $\partial_x\left(\textstyle\frac{2}{r\rho}u\right)P_x$ and $\partial_t^k\partial_x\left[\frac{2}{r}\left(P+1\right)u\right]\partial_t^k\partial_xP$ would be hard to cope with. So to avoid the emergence of $\partial_x^kr$, we treat $(r^2u)$ as a whole and this strategy will be utilized multiple times in the following content \ep{refer to \eqref{barq} and \eqref{UkQk}}.
\end{remark}

 \section{local existence}\label{local}
In this section, we shall establish the local existence of the solution to initial boundary problem \eqref{cns5},\eqref{initial},\eqref{bdry} by iteration method. The main results are presented as follows.
\begin{theorem}\label{localexistence}
	If the initial data satisfies \eqref{initialdata} and compatible conditions \eqref{compatible}, then there exist some $T_0>0$ and $K_3\{\epsilon_0\}>0$ such that the system \eqref{cns5},\eqref{initial},\eqref{bdry} admit a unique classical solution
	$\left(P,u,s,q\right)$, which satisfies the regularity that
	\begin{align}
		(P,u,s,q,q_x)\in B_m\left([0,T_0];\mathbb{I};K_3\{\epsilon_0\}\right).
	\end{align}
\end{theorem}
In the beginning, we provide a helpful lemma for later use.
\begin{lemma}[\text{Theorem 2.5.7 in \cite[p.~55]{MR0161012}}]\label{giveninitialvalues}
	For arbitrary $f_k(\bm{x})\in H^{m-k-\frac 12}(\mathbb{R}^n)$, $k=0,\cdots,m-1$, there exists a function $u(t,\bm{x})\in H^m([0,+\infty)\times\mathbb{R}^n)$ with $\partial_t^ku(0,\bm{x})=f_k(\bm{x})$, $k=0,\cdots,m-1$.
\end{lemma}

In order to construct the successive approximation sequence in the iteration process, we need to first consider the linearized system.
\subsection{Solution to the linearized system}
The desired linearized system takes the form as
\begin{equation}
\label{cns9}
\left\{	
\begin{aligned}
&P_t+\frac{C_v+1}{C_v}(\widebar{P}+1)\bar{\rho}\bar{r}^2u_x+\frac{C_v+1}{C_v}(\widebar{P}+1)\bar{\rho}2\bar{r}\bar{u}u+\frac{1}{C_v}\bar{\rho}(\bar{r}^2q)_x=0,\\[2mm]
&u_t+\bar{r}^2P_x=0,\\[2mm]
&s_t+\frac{(\bar{r}^2q)_x}{\raisebox{-0.5mm}{$\bar{\theta}$}}=0,\\
&q-\bar{r}^2\bar{\rho}^2(\bar{r}^2q)_{xx}+\frac{4\bar{r}^2\bar{\theta}^3}{C_v+1}\left(\widebar{P}_x+\bar{s}_x\right)=\frac 1 2\bar{r}^2(\bar{\rho}^2)_x(\bar{r}^2q)_x-\frac{4\bar{r}^2\bar{\theta}^3}{C_v+1}\widebar{P}\bar{s}_x,
\end{aligned}	
\right.
\end{equation}
with corresponding initial boundary conditions that
\begin{alignat}{2} \label{linitial}
(P,u,s)(0,x)&=
(P_0-1,u_0,s_0-1)(x),
\qquad && x\in\mathbb{I},\\[1mm]               \label{lbdry}
(u,q)(t,0)&=(u,q)(t,1)=0,\qquad&& t\geq 0,
\end{alignat}
Here $\widebar{P}(t,x),\ \bar{u}(t,x),\ \bar{s}(t,x)$ are given and
\begin{equation}\label{bar}
\left\{
\begin{aligned}
&\left|\left(\widebar{P},\bar{u},\bar{s}\right)\right|^2_{L^\infty}\lesssim\sup_{0\leq t\leq T}\left[\vertiii{\left(\widebar{P},\bar{u},\bar{s}\right)(t)}^2_{m-1}+\vertiii{\left(\widebar{P},\bar{u},\bar{s}\right)(t)}^2_{m,tan}\right]\leq K_1\{\epsilon_0\},\\[2mm]
&\vertiii{\left(\widebar{P},\bar{u},\bar{s}\right)}_{m,T}\leq K_2\{\epsilon_0\}.
\end{aligned}
\right.
\end{equation}
It is clear that $\bar{\rho},\bar{\theta}$ are also known, since they can be expressed by $\widebar{P},\bar{s}$ \ep{refer to \eqref{transform}}. We define $\bar{r}$ as that in \eqref{coordinate}.
\begin{align}\label{barrdefinition}
	\bar{r}=r_0(x)+\int_{0}^t\bar{u}(\tau,x){\rm d}\tau.
\end{align}
Especially, we need to point it out that other than \eqref{r_eq}, the derivatives of $\bar{r}$ take the form that
\begin{align*}
	\bar{r}_t=\bar{u},\qquad\bar{r}_x=\frac{1}{r_0^2\rho_0}+\int_0^t\bar{u}_x(\tau,x)\ {\rm d}\tau.
\end{align*}
The reason is that \eqref{r_eq} comes from the nonlinear mass-conservation equation \eqref{mass-conservation} and definition \eqref{coordinate}, but the former one is not valid in linearized system. With \eqref{bar}, we can see that $\bar{r}$ has positive upper and lower bounds when $\epsilon_0$, $K_1\{\epsilon_0\}$, $T$ are small and fixed. Moreover $\vertiii{\bar{r}}_{m,T}$ has a upper bound concerning $\epsilon_0$, $\mathbb{I}$, $T$ and $K_2\{\epsilon_0\}$.
\begin{align}\notag
	\vertiii{\bar{r}(t)}_{m}\leq& \left\|r_0\right\|_{H^{m+1}}\cdot\big\lvert\,\mathbb{I}\,\big\lvert+\vertiii{u(t)}_{m-1}+T\cdot\vertiii{u}_{m,T}\\\label{barr}
	\leq &\left\|r_0\right\|_{H^{m+1}}\cdot\big\lvert\,\mathbb{I}\,\big\lvert+\vertiii{u(0)}_{m-1}+2T\cdot\vertiii{u}_{m,T}
	\leq C_2\{\epsilon_0,\mathbb{I}\}+2T\cdot K_2\{\epsilon_0\}.
\end{align}
\ep{Here, $\big\lvert\,\mathbb{I}\,\big\lvert$ denotes the measure of domain $\mathbb{I}$. Since our region $\mathbb{I}$ is fixed, it won't appear afterwards in our proof as a variable.} To prove \eqref{barr}, since $\bar{r}_t=\bar{u}$, we just need to consider the case when only spatial derivatives are involved and with Minkowski's integral inequality, it can be easily proved. The detailed proof is omitted here. Now we start to work on the existence of solutions to such linearized system and the main results are presented in the following Theorem.
\begin{theorem}[solution to the linearized system]\label{linearizedsolution}
	The initial boundary problem \eqref{cns9}--\eqref{lbdry} has a unique solution
	$(P,u,s,q)$ satisfying $(P,u,s,q,q_x)\in X_m\left([0,T];\mathbb{I}\right)$ for any given $T>0$, which is provided that 	
	\begin{itemize}
		\item[\ep{i}] The known functions $\left(\widebar{P},\bar{u},\bar{s}\right)$ satisfy \eqref{bar};
		\item[\ep{ii}] $F\in X_m([0,T],\mathbb{I})$ and $(P_0-1,u)\in H^m(\mathbb{I})$;
		\item[\ep{iii}] The compatibility conditions that $\partial_t^ku(0,x)=0$ on $\partial\mathbb{I}$ hold.
	\end{itemize}
\end{theorem}

First, consider the last equation in \eqref{cns9}. It is a standard second-order linear elliptic equation with Dirichlet boundary conditions provided that $K_2\{\epsilon_0\}$ in \eqref{bar} is sufficiently small. The solution to such problem is quite clear (e.g. see for chapter 6 in \cite{MR2597943}) and satisfies that
\begin{align}\label{barq}
	 \vertiii{\Big(q,(\bar{r}^2q)_x\Big)(t)} _{m}^2\leq G_1\big\{\epsilon_0,T,K_1\{\epsilon_0\}\big\}\vertiii{\left(\widebar{P},\bar{s}\right)(t)}^2_m.
\end{align}
The estimates \eqref{barq} can be obtained just like what we did in the a priori estimates with respect to $q$ in the nonlinear system. One thing we want to point out here is that due to the definition \eqref{barrdefinition}, $\partial_x^{m+1}\bar{u}$ would appear when we compute $\partial_x^{m}(\bar{r}^2q)_x$, which may not be bounded in $L^2-$norm. And that is why we treat $\bar{r}^2q$ as a whole.
Now since $q$ is solved, $s$ can be easily obtained by integrate \eqref{cns9} over $[0,t]$, which is
\begin{align*}
s(t,x)=s_0(x)+\int_{0}^{t}\frac{(\bar{r}^2q)_x}{\raisebox{-0.5mm}{$\bar{\theta}$}}(s,x)\ {\rm d}s.
\end{align*}
With clear expression of $s$, we can get the following estimates that
\begin{align}\notag
&\vertiii{s(t)}^2_{m,tan}\leq	\left\|\partial_t^{m-1}\left(\frac{(\bar{r}^2q)_x}{\raisebox{-0.5mm}{$\bar{\theta}$}}\right)(t)\right\|^2
\lesssim\vertiii{\frac{(\bar{r}^2q)_x}{\raisebox{-0.5mm}{$\bar{\theta}$}}(t)}_{m-1}^2\ \ \
\mathrel{\overset{\makebox[0pt]{\mbox{\scriptsize\sffamily Lemma \ref{inequality}\ep{1}}}}{\lesssim}}
\qquad
\vertiii{\frac{1}{\raisebox{-0.5mm}{$\bar{\theta}$}}(t)}^2_{m-1}\vertiii{(\bar{r}^2q)_x(t)}^2_{m-1}\\\notag
&\qquad\qquad\,
\lesssim\left(\vertiii{\frac{1}{\raisebox{-0.5mm}{$\bar{\theta}$}}(0)}^2_{m-1}
+T\cdot\vertiii{\frac{1}{\raisebox{-0.5mm}{$\bar{\theta}$}}}_{m,T}^2\right)
\left[\vertiii{(\bar{r}^2q)_x(0)}^2_{m-1}+T\cdot\vertiii{(\bar{r}^2q)_x}^2_{m,T}\right]\\\label{smtan}
&\qquad\qquad\,
\leq\Big(G_{2}\{\epsilon_0\}+T\cdot G_{3}\big\{K_2\{\epsilon_0\}\big\}\Big)
\Big(G_{4}\{\epsilon_0\}\epsilon_0
+T\cdot G_1\big\{\epsilon_0,T,K_1\{\epsilon_0\}\big\}\left(K_2\{\epsilon_0\}\right)^2\Big),\\[2mm]\notag
&\vertiii{s(t)}^2_{k}\lesssim
\left\|s_0\right\|^2_{H^m}\!\!+
\vertiii{\frac{1}{\raisebox{-0.5mm}{$\bar{\theta}$}}}^2_{m-1}\!\!\vertiii{(\bar{r}^2q)_x}^2_{m-1}\!\!+\!\!
\int_{0}^t \vertiii{\frac{1}{\raisebox{-0.5mm}{$\bar{\theta}$}}(\tau)}^2_{m}\!\!\vertiii{(\bar{r}^2q)_x(\tau)}^2_{m}{\rm d}\tau.
\qquad(k=m-1,m)\\\notag
&\qquad\quad\
\lesssim\epsilon_0+\Big(G_{2}\{\epsilon_0\}+T\cdot G_{3}\big\{K_2\{\epsilon_0\}\big\}\Big)
\Big(G_{4}\{\epsilon_0\}\epsilon_0
+T\cdot G_1\big\{\epsilon_0,T,K_1\{\epsilon_0\}\big\}\left(K_2\{\epsilon_0\}\right)^2\Big)\\\label{sv}
&\qquad\quad\quad\ \
+T\cdot G_{5}\big\{K_2\{\epsilon_0\}\big\}G_1\big\{\epsilon_0,T,K_1\{\epsilon_0\}\big\}\left(K_2\{\epsilon_0\}\right)^2
\end{align}
As for $P$ and $u$, we take the first two equations in \eqref{cns9} as a individual system. Rewrite them in the form that
\begin{equation}\label{Pu}
\left(  \begin{matrix}
	\frac{C_v}{(C_v+1)\left(\widebar{P}+1\right)\bar{\rho}\bar{r}^2}&0\\
	0&\frac 1 {\bar{r}^2}
  \end{matrix}
  \right)\frac{\partial}{\partial_t}
\left(\begin{matrix}
P\\
u
\end{matrix}\right)+
\left(\begin{matrix}
 0&1\\
 1&0
\end{matrix}\right)\frac{\partial} {\partial_x}
\left(\begin{matrix}
P\\
u
\end{matrix}\right)+
\left(\begin{matrix}
0&\frac{2\bar{u}}{\bar{r}}\\
0&0
\end{matrix}\right)
\left(\begin{matrix}
P\\
u
\end{matrix}\right)=
\left(\begin{matrix}
-\frac{(\bar{r}^2q)_x}{(C_v+1)(\widebar{P}+1)\bar{r}^2}\\
0
\end{matrix}\right)
\end{equation}
Now for $U=(P,u)^T$, we define the corresponding linear partial differential operator that
\begin{align*}
	LU=AU_t+BU_x+CU,\qquad \text{where}\
	A=\left(  \begin{matrix}
	\frac{C_v}{(C_v+1)\left(\widebar{P}+1\right)\bar{\rho}\bar{r}^2}&0\\
	0&\frac 1 {\bar{r}^2}
	\end{matrix}
	\right),\
	B=\left(\begin{matrix}
	0&1\\
	1&0
	\end{matrix}\right),\
	C=\left(\begin{matrix}
	0&\frac{2\bar{u}}{\bar{r}}\\
	0&0
	\end{matrix}\right).
\end{align*}
Since $F:=\left(-\frac{(\bar{r}^2q)_x}{(C_v+1)(\widebar{P}+1)\bar{r}^2},0\right)^T$ is given, we can see that $LU=F$ is a linear symmetric hyperbolic system with initial boundary conditions
\begin{alignat}{2} \label{linitial3}
(P,u)(0,x)&=
(P_0-1,u_0)(x),
\qquad && x\in\mathbb{I},\\[1mm]               \label{lbdry3}
u(t,0)&=u(t,1)=0,\qquad&& t\geq 0.
\end{alignat}
The boundary condition \eqref{lbdry3} is maximally nonnegative according to \cite[p.~62]{MR834481} (or maximal dissipative refer to Definition 2.1.3 in \cite[p.~22]{MR2151414}). The local existence of solution to the initial boundary problem \eqref{Pu}-\eqref{lbdry3} is proved in the following theorem
\begin{lemma}[solution to the hyperbolic part]
	The initial boundary problem $LU=F$ with \eqref{linitial3} and \eqref{lbdry3} has a unique solution in $X_m([0,T],\mathbb{I})\quad(m\geq2)$, provided that 	
		\begin{itemize}
		\item[\ep{i}] The known functions $\left(\widebar{P},\bar{u},\bar{s}\right)$ satisfy \eqref{bar};
		\item[\ep{ii}] $F\in X_m([0,T],\mathbb{I})$ and $(P_0-1,u)\in H^m(\mathbb{I})$;
		\item[\ep{iii}] The compatibility conditions that $\partial_t^ku(0)=0$ on $\partial\mathbb{I}$ hold.
		\end{itemize}
The solution $U=(P,u)^T$ obeys the estimates that
\begin{align}\notag
&\vertiii{(P,u)(t)}_{m-1}^2+ \vertiii{(P,u)(t)}_{m,tan}^2
\leq \\[3mm]\label{put}
	&\qquad\quad\begin{multlined}
\Big(G_{6}\{\epsilon_0\}+T\cdot G_{7}\big\{\epsilon_0,T,K_1\{\epsilon_0\}\big\}\Big)
\exp\Big\{{t\cdot G_{8}\big\{\epsilon_0,T,K_2\{\epsilon_0\}\big\}}\Big\}
\left[\vertiii{(P,u)(0)}_{m-1}^2
\right.\\
+\left.
\vertiii{(P,u)(0)}_{m,tan}^2+
G_{9}\big\{\epsilon_0,T,K_2\{\epsilon_0\}\big\}\int_{0}^{t} \vertiii{(\bar{r}^2q)_x(s)}_{m}^2{\rm d}s\right],
	\end{multlined}
    	\\[3mm]\notag
	&\vertiii{\left(P,u\right)(t)}_m\leq G_{10}\big\{\epsilon_0,T,K_1\{\epsilon_0\}\big\}
	\left(1+\vertiii{\left(\widebar{P},\bar{u},\bar{s},\bar{r}\right)}_{m,T}^a\right)\cdot\\[1.5mm] \label{puv}
	&\qquad\qquad\qquad\qquad\qquad\qquad\quad\left[\vertiii{(P,u)(t)}_{m-1}
	+\vertiii{(P,u)(t)}_{m,tan}+\vertiii{(\bar{r}^2q)_x(t)}_{m-1}\right].
\end{align}
Here, constant $a$ satisfies $0<a<1$.
\begin{proof}
	As for the existence and uniqueness of solution, one can refer to Theorem A1 in Appendix A of \cite{MR834481} and the important reference \cite[Theorem 3.1]{MR340832} in it. Now it is left to prove \eqref{put} and \eqref{puv}, which can be derived from the following estimates.
	\begin{align}\notag
	&\frac{{\rm d}}{{\rm d}t} \int_{\mathbb{I}}\sum_{\substack{|\alpha|\leq m\\|\alpha|\leq m-1\ \text{or}\ \alpha_2=0}}
	\left[\frac{C_v}{2(C_v+1)\left(\widebar{P}+1\right)\bar{\rho}\bar{r}^2}\left(D^{\alpha}P\right)^2
	+\frac{1}{2\bar{r}^2}\left(D^{\alpha}u\right)^2\right]{\rm d}x\\\label{m-1mtan}
	&\qquad\qquad\qquad\quad\ \leq C\big\{\epsilon_0,T,K_2\{\epsilon_0\}\big\}\vertiii{(P,u)}^2_m+C\left\{\epsilon_0,T,K_2\{\epsilon_0\}\right\}\vertiii{\Big(q,(\bar{r}^2q)_x\Big)}_m,\\[4mm]\notag
	&\left\|\partial_t^{m-i}\partial_x^{i}\left(P,u\right)\right\|\leq C\big\{\epsilon_0,T,K_1\{\epsilon_0\}\big\}\left\|\partial_t^{m-i+1}\partial_x^{i-1}\left(P,u\right)\right\|
	+\epsilon_1\textstyle\sum_{i=1}^{m}\left\|\partial_t^{m-i}\partial_x^{i}\left(P,u\right)\right\|+\\[1.5mm]\notag
&\qquad\qquad\qquad\quad\ \  C\big\{\epsilon_0,T,K_1\{\epsilon_0\}\big\}
\left(1+\vertiii{\left(\widebar{P},\bar{u},\bar{s},\bar{r}\right)}_{m,T}^a\right)
\left[\vertiii{(P,u)(t)}_{m-1}
+\vertiii{(P,u)(t)}_{m,tan}\right.\\[1.5mm]\label{im-i}
&\qquad\qquad\qquad\quad\ \left.+\vertiii{\Big(q,(\bar{r}^2q)_x\Big)(t)}_{m-1}\right]\quad\left(1\leq i\leq m\right),\quad \text{for a small constant }\epsilon_1.
	\end{align}
	First, we will explain how to obtain \eqref{put}, \eqref{puv} by using \eqref{m-1mtan} and \eqref{im-i}. According to \eqref{im-i}, we can see that when $i=1$, $\left\|\partial_t^{m-1}\partial_x^{1}\left(P,u\right)\right\|$ is bounded by the right side of \eqref{puv} except for the term $\epsilon_1\textstyle\sum_{i=1}^{m}\left\|\partial_t^{m-i}\partial_x^{i}\left(P,u\right)\right\|$. Substitute it into \eqref{im-i} when $i=2$ and we have $\left\|\partial_t^{m-2}\partial_x^{2}\left(P,u\right)\right\|$ also does. Iteratively, it holds for all $1\leq i\leq m$. Add them up and we have $\sum_{i=1}^{m}\left\|\partial_t^{m-i}\partial_x^{i}\left(P,u\right)\right\|$ bounded by the right side of \eqref{puv}. The remaining terms in $\vertiii{(P,u)(t)}_m$ naturally satisfy \eqref{puv}, if we set $G_{10}\big\{\epsilon_0,T,K_1\{\epsilon_0\}\big\}>1$. To prove\eqref{put}, we substitute \eqref{puv} into the right side of \eqref{m-1mtan}, then use
	\begin{align}\label{1}
        \left|\frac{C_v}{2(C_v+1)\left(\widebar{P}+1\right)\bar{\rho}\bar{r}^2}\right|_{L^{\infty}},\ \left|\frac{1}{2\bar{r}^2}\right|_{L^{\infty}}
        &\leq C\{\epsilon_0\}+
        T\cdot C\big\{\epsilon_0,T,K_2\{\epsilon_0\}\big\},\\\label{-1}
		\left|\left(\frac{C_v}{2(C_v+1)\left(\widebar{P}+1\right)\bar{\rho}\bar{r}^2}\right)^{-1}\right|_{L^{\infty}},\ \left|\left(\frac{1}{2\bar{r}^2}\right)^{-1}\right|_{L^{\infty}}
		&\leq C\{\epsilon_0\}+
		T\cdot C\big\{\epsilon_0,T,K_2\{\epsilon_0\}\big\},
	\end{align}
	and $Gr\ddot{o}nwall's$ inequality and eventually we have \eqref{put}. To prove \eqref{1} here, we can see
	\begin{align*}
	\left\|\frac{C_v}{\left(C_v+1\right)\left(\widebar{P}+1\right)\bar{\rho}\bar{r}^2}\right\|_{C^k\left(\widebar{P},\bar{\rho},\bar{r}\right)}\leq C\big\{\epsilon_0,T,K_1\{\epsilon_0\}\big\}.
	\end{align*}
	And \eqref{1} follows from
	\begin{align*}
		&\quad\ \left|\frac{C_v}{2(C_v+1)\left(\widebar{P}+1\right)\bar{\rho}\bar{r}^2}\right|_{L^{\infty}},\ \left|\frac{1}{2\bar{r}^2}\right|_{L^{\infty}}
		\lesssim \sup_{0\leq t\leq T}
		\vertiii{\left(\frac{C_v}{2(C_v+1)\left(\widebar{P}+1\right)\bar{\rho}\bar{r}^2}\ ,\ \frac{1}{2\bar{r}^2}\right)(t)}_{m-1}\\
		&\lesssim\vertiii{\left(\frac{C_v}{2(C_v+1)\left(\widebar{P}+1\right)\bar{\rho}\bar{r}^2}\ ,\ \frac{1}{2\bar{r}^2}\right)(0)}_{m-1}+
		T\cdot\vertiii{\left(\frac{C_v}{2(C_v+1)\left(\widebar{P}+1\right)\bar{\rho}\bar{r}^2}\ ,\ \frac{1}{2\bar{r}^2}\right)}_{m,T}\\
		&\qquad\qquad\qquad\qquad\qquad\qquad\qquad\qquad\qquad
		\mathrel{\overset{\makebox[0pt]{\mbox{\scriptsize\sffamily Lemma \ref{inequality}\ep{3}}}}{\leq}}\qquad
		C\{\epsilon_0\}+
		T\cdot C\big\{\epsilon_0,T,K_2\{\epsilon_0\}\big\}.
	\end{align*}
Similarly, we can derive \eqref{-1}.	

	Second, we aim to prove \eqref{m-1mtan}. Apply $\partial_t^k(0\leq k\leq m)$ on \eqref{Pu}, multiply the resultant equality by $(\partial_t^kP,\partial_t^ku)$ and integrate over $\mathbb{I}$ to get
	\begin{align}\notag
	\frac{{\rm d}}{{\rm d}t}&
	 \int_{\mathbb{I}}\left[\frac{C_v}{2(C_v+1)\left(\widebar{P}+1\right)\bar{\rho}\bar{r}^2}\left(\partial_t^kP\right)^2
	+\frac{1}{2\bar{r}^2}\left(\partial_t^ku\right)^2\right]{\rm d}x\\\notag
	=&\int_{\mathbb{I}}\frac{{\rm d}}{{\rm d}t}\left[\frac{C_v}{2(C_v+1)\left(\widebar{P}+1\right)\bar{\rho}\bar{r}^2}\right]\left(\partial_t^kP\right)^2
	+\underbrace{\left[\partial_t^k, \frac{C_v}{\left(C_v+1\right)\left(\widebar{P}+1\right)\bar{\rho}\bar{r}^2}\right]\partial_tP\cdot\partial_t^kP}_{\circled{3}}+\\\notag
	&\partial_t^k\left(\frac{2\bar{u}}{\bar{r}}u\right)\cdot\partial_t^kP+\partial_t^k\left(\frac{(\bar{r}^2q)_x}{\left(C_v+1\right)(\widebar{P}+1)\bar{r}^2}\right)\cdot\partial_t^kP\ {\rm d}x
	+\int_{\mathbb{I}}\frac{{\rm d}}{{\rm d}t}\left(\frac{1}{2\bar{r}^2}\right)\left(\partial_t^ku\right)^2+ \\\notag
	&\left[\partial_t^k,\frac{1}{\bar{r}^2}\right]\partial_tu\cdot\partial_t^ku\ {\rm d}x\\\label{mtan}
	&\leq C\big\{\epsilon_0,T,K_2\{\epsilon_0\}\big\}\vertiii{(P,u)(t)}^2_m+C\big\{\epsilon_0,T,K_2\{\epsilon_0\}\big\}\vertiii{(\bar{r}^2q)_x(t)}_m.
	\end{align}
	We use $\circled{3}$ to illustrate how the estimates are done.
	\begin{align*}
		\int_{\mathbb{I}}\circled{3}\ {\rm d}x\leq&\left\|\partial_t^kP\right\|\sum_{\substack{|\alpha|+|\beta|=k\\1\leq|\alpha|\leq k\\0\leq|\beta|\leq k-1}}
		\left\|D^\alpha\left(\frac{C_v}{\left(C_v+1\right)\left(\widebar{P}+1\right)\bar{\rho}\bar{r}^2}\right)D^\beta\left(\partial_tP\right)\right\|\\
		\leq&C\big\{\epsilon_0,T,K_2\{\epsilon_0\}\big\}\vertiii{P}^2_m.
    \end{align*}
    Here,
    \begingroup
    \footnotesize
    \begin{align*}
		&\left\|D^\alpha\left(\frac{C_v}{\left(C_v+1\right)\left(\widebar{P}+1\right)\bar{\rho}\bar{r}^2}\right)D^\beta\left(\partial_tP\right)\right\|\\[3mm]
		&\ \ \lesssim\left\{
		\begin{aligned}
		&\left|D\left(\frac{C_v}{\left(C_v+1\right)\left(\widebar{P}+1\right)\bar{\rho}\bar{r}^2}\right)\right|_{L^{\infty}}\vertiii{\partial_tP}_{m-1}
		\lesssim\vertiii{\frac{C_v}{\left(C_v+1\right)\left(\widebar{P}+1\right)\bar{\rho}\bar{r}^2}}_2\vertiii{\partial_tP}_{m-1},\ \begin{aligned}
		&|\alpha|=1\\&|\beta|=k-1
		\end{aligned}\\[5mm]
		&\vertiii{\frac{C_v}{\left(C_v+1\right)\left(\widebar{P}+1\right)\bar{\rho}\bar{r}^2}}_m\left|\partial_tP\right|_{L^{\infty}}
		\lesssim\vertiii{\frac{C_v}{\left(C_v+1\right)\left(\widebar{P}+1\right)\bar{\rho}\bar{r}^2}}_m\vertiii{\partial_tP}_1,\qquad\ \begin{aligned}
		&|\alpha|=k\\&|\beta|=0
		\end{aligned}\\[5mm]
		&\vertiii{\frac{C_v}{\left(C_v+1\right)\left(\widebar{P}+1\right)\bar{\rho}\bar{r}^2}}_m\vertiii{\partial_tP}_{m-1}\ (\text{with Lemma \ref{inequality} (2)}),\qquad\begin{aligned}
		&|\alpha|+|\beta|<k+1\leq 2m-1\\
		&|\alpha|\leq k-1\leq m-1\\
		&|\beta|\leq k-2\leq m-2
		\end{aligned}
		\end{aligned}
		\right.\\[5mm]
		&\ \ \lesssim\vertiii{\frac{C_v}{\left(C_v+1\right)\left(\widebar{P}+1\right)\bar{\rho}\bar{r}^2}}_m\vertiii{P}_{m}\quad
		\mathrel{\overset{\makebox[0pt]{\mbox{\tiny\sffamily Lemma \ref{inequality}\ep{3}}}}{\leq}}\quad\  C\big\{\epsilon_0,T,K_2\{\epsilon_0\}\big\}\vertiii{P}_m.
	\end{align*}
\endgroup
  The remaining terms on the left side of \eqref{put} are $\vertiii{(P,u)(t)}_{m-1}$, the estimates for which can be obtained by following the same procedure but omitting the integration by parts in the spatial derivative terms to avoid boundary terms. With these estimates and \eqref{mtan}, the proof of \eqref{m-1mtan} completes.

Third, we are going to derive \eqref{puv}. Apply $\partial_t^{m-i}\partial_x^i$ on \eqref{Pu} and solve it for $\partial_t^{m-i}\partial_x^iP$, $\partial_t^{m-i}\partial_x^iu$ to get
\begin{align}\notag
	&\left\|\partial_t^{m-i}\partial_x^i\left(P,u\right)\right\|
	\leq\left|\frac{C_v}{\left(C_v+1\right)\left(\widebar{P}+1\right)\bar{\rho}\bar{r}^2}\right|_{L^\infty}\left\|\partial_t^{m-i+1}\partial_x^{i-1}P\right\|
	+\left|\frac{1}{\bar{r}^2}\right|_{L^\infty}\left\|\partial_t^{m-i+1}\partial_x^{i-1}u\right\|\\\notag
	&+\sum_{\substack{|\alpha|+|\beta|=m-1\\1\leq|\alpha|\leq m-1\\0\leq|\beta|\leq m-2}}
	\left\|D^\alpha\left(\frac{C_v}{\left(C_v+1\right)\left(\widebar{P}+1\right)\bar{\rho}\bar{r}^2}\right)D^\beta\left(\partial_tP\right)\right\|
	+\sum_{\substack{|\alpha|+|\beta|=m-1\\1\leq|\alpha|\leq m-1\\0\leq|\beta|\leq m-2}}
	\left\|D^\alpha\left(\frac{1}{\bar{r}^2}\right)D^\beta\left(\partial_tu\right)\right\|\\\notag
	&+\left\|\partial_t^{m-i}\partial_x^{i-1}\left(\frac{2\bar{u}}{\bar{r}}u\right)\right\|
	+\left\|\partial_t^{m-i}\partial_x^{i-1}\left[\frac{(\bar{r}^2q)_x}{\left(C_v+1\right)(\widebar{P}+1)\bar{r}^2}\right]\right\|\\[3mm] \notag
	&\leq C\big\{\epsilon_0,T,K_1\{\epsilon_0\}\big\}\left\|\partial_t^{m-i+1}\partial_x^{i-1}\left(P,u\right)\right\|+
	\epsilon_1\textstyle\sum_{i=1}^{m}\left\|\partial_t^{m-i}\partial_x^{i}\left(P,u\right)\right\|+\\[1.5mm]\label{m-ii}
	&\quad C\big\{\epsilon_0,T,K_1\{\epsilon_0\}\big\}
	\left(1+\vertiii{\left(\widebar{P},\bar{u},\bar{s},\bar{r}\right)}_{m,T}^a\right)
	\left(\vertiii{P}_{m-1}+\vertiii{P}_{m,tan}+\vertiii{(\bar{r}^2q)_x(t)}_{m-1}\right),
\end{align}
and \eqref{puv} is proved. We use the following estimate to illustrate why the last inequality in \eqref{m-ii} stands.
\begin{align*}\notag
    &\sum_{\substack{|\alpha|+|\beta|=m-1\\1\leq|\alpha|\leq m-1\\0\leq|\beta|\leq m-2}}
    \left\|D^\alpha\left(\frac{C_v}{\left(C_v+1\right)\left(\widebar{P}+1\right)\bar{\rho}\bar{r}^2}\right)D^\beta\left(\partial_tP\right)\right\|\\\notag
&\mathrel{\overset{\makebox[0pt]{\mbox{\scriptsize\sffamily Lemma \ref{inequality}\ep{2}}}}{\lesssim}}\qquad
\vertiii{\frac{C_v}{\left(C_v+1\right)\left(\widebar{P}+1\right)\bar{\rho}\bar{r}^2}}_{m-1}^{1-a_1}
\vertiii{\frac{C_v}{\left(C_v+1\right)\left(\widebar{P}+1\right)\bar{\rho}\bar{r}^2}}_{m}^{a_1}
\vertiii{P}_{m-1}^{1-a_2}\vertiii{P}_{m}^{a_2}\\\notag
&\mathrel{\overset{\makebox[0pt]{\mbox{\scriptsize\sffamily small $\epsilon_2$}}}{\lesssim}}\qquad
C\big\{\epsilon_0,T,K_1\{\epsilon_0\}\big\}\vertiii{\frac{C_v}{\left(C_v+1\right)\left(\widebar{P}+1\right)\bar{\rho}\bar{r}^2}}_{m}
\left(\epsilon_2\vertiii{P}_{m}+C\vertiii{P}_{m-1}\right)\\\notag
&\mathrel{\overset{\makebox[0pt]{\mbox{\scriptsize\sffamily Lemma \ref{inequality}\ep{3}}}}{\leq}}\qquad
C\big\{\epsilon_0,T,K_1\{\epsilon_0\}\big\}
\left(1+\vertiii{\left(\widebar{P},\bar{u},\bar{s},\bar{r}\right)}_m^{a_1}\right)
\left[\epsilon_2\textstyle\sum_{i=1}^{m}\left\|\partial_t^{m-i}\partial_x^{i}\left(P,u\right)\right\|+\right.\\[1.5mm]\notag
&\qquad\quad\left.C\left(\vertiii{P}_{m-1}+\vertiii{P}_{m,tan}\right)\right]\\\notag
&\leq\epsilon_1\textstyle\sum_{i=1}^{m}\left\|\partial_t^{m-i}\partial_x^{i}\left(P,u\right)\right\|+C\big\{\epsilon_0,T,K_1\{\epsilon_0\}\big\}
\left(1+\vertiii{\left(\widebar{P},\bar{u},\bar{s},\bar{r}\right)}_{m,T}^a\right)
\left(\vertiii{P}_{m-1}+\right.\\[2mm]
&\quad \left.\vertiii{P}_{m,tan}\right).\qquad
\left(\text{Let } a=a_1,\epsilon_1=\epsilon_2C\big\{\epsilon_0,T,K_1\{\epsilon_0\}\big\}
		\left(1+\vertiii{\left(\widebar{P},\bar{u},\bar{s},\bar{r}\right)}_{m,T}^a\right).\right)
\end{align*}
Now the whole proof is accomplished.

\end{proof}
\end{lemma}
\subsection{Iteration Scheme}
Now with the solvability of linearized system, we are able to prove Theorem \ref{thm} through the classical iteration process, which can be roughly presented as obtaining a solution-sequence  $\left([P_{k+1}],[u_{k+1}],[s_{k+1}],[q_{k+1}]\right)$ iteratively to the following linear system, for $k=0,1,2,\cdots$.
\begin{equation}
\label{cns10}
\left\{	
\begin{aligned}
&[P_{k+1}]_t+\frac{C_v+1}{C_v}([P_k]+1)[\rho_k]\left([r_k]^2[u_{k+1}]\right)_x+\frac{1}{C_v}[\rho_k]([r_k]^2[q_{k+1}])_x=0,\\[2mm]
&[u_{k+1}]_t+[r_k]^2[P_{k+1}]_x=0,\\[2mm]
&[s_{k+1}]_t+\frac{([r_k]^2[q_{k+1}])_x}{[\theta_k]}=0,\\
&\frac{[q_{k+1}]}{[r_k]^2[\rho_k]}-\Big([\rho_k]([r_k]^2[q_{k+1}])_{x}\Big)_x+\frac{4[\theta_k]^3}{(C_v+1)[\rho_k]}\Big([P_k]_x+\left([P_k]+1\right)[s_k]_x\Big)=0 .
\end{aligned}	
\right.
\end{equation}
with
\begin{alignat}{2} \label{kinitial}
([P_{k+1}],[u_{k+1}],[s_{k+1}])(0,x)&=
(P_0-1,u_0,s_0-1)(x),
\qquad && x\in\mathbb{I},\\[1mm]               \label{kbdry}
([u_{k+1}],[q_{k+1}])(t,0)&=([u_{k+1}],[q_{k+1}])(t,1)=0,\qquad&& t\geq 0,
\end{alignat}
At outset, it is not obvious the iterates above are well-posed. The following three lemmas are aimed to establish this. First, we can see that the starting point of iteration is not given yet and the iteration systems do not necessarily satisfy the compatible conditions. Usually, the initial data (see \cite[p.~36]{MR0748308}) or the equilibrium states (see \cite[p.~307]{MR2022134}) are used for the beginning of iteration. However, in our case when boundary condition are included these two won't work since we need all iteration systems satisfy the compatible conditions. We can see it in Lemma that if the starting point we choose such that the first system satisfies the compatible conditions, then all systems in iteration sequence  naturally do. So the fact we aim to prove in the first lemma is that such starting point does exist, which is analogous to \cite[Lemma A3]{MR834481}.
\begin{lemma}\label{lemma1}
\begin{spacing}{1.3}
There exist $\left([P_{0}],[u_{0}],[s_{0}]\right)\in X_m([0,T];\mathbb{I})$ which satisfy the boundary conditions $[u_{0}](t,0)=[u_{0}](t,1)=0$ and initial conditions
\begin{equation}\label{initialconditions}
	\left(\partial_t^k[P_{0}],\partial_t^k[u_{0}],\partial_t^k[s_{0}]\right)(0,x)=
	\left\{
	\begin{aligned}
	&(P_0-1,u_0,s_0-1), &&k=0\\
	&\left(\partial_t^kP,\partial_t^ku,\partial_t^ks\right)(0),&\quad&1\leq k\leq m
	\end{aligned}
	\right.
\end{equation}
\end{spacing}
\begin{proof}
We just need to find some functions $\left([P_{0}],[u_{0}],[s_{0}]\right)\in H^{m+1}([0,T]\times\mathbb{I})$,
which satisfy the initial conditions \eqref{initialconditions}. To achieve that, we can utilize Lemma \ref{giveninitialvalues}, which requires the values that $\left(\partial_t^k[P_{0}],\partial_t^k[u_{0}],\partial_t^k[s_{0}]\right)(0,x)$ take should belong to $H^{m+1-k-\frac 1 2}(\mathbb{I})$, if we want $\left([P_{0}],[u_{0}],[s_{0}]\right)\in H^{m+1}([0,T]\times\mathbb{I})$. It is a natural requirement if we see from Trace Theorem. However, the values in \eqref{initialconditions} only belong to $H^{m-k}(\mathbb{I})$. To overcome this difficulty, we shall make use of our linearized system. Let $\left([P_{k}],[u_{k}],[s_{k}]\right)=(P_0,u_0,s_0)$ and $\left([\rho_k],[\theta_k],[r_k]\right)=(\rho_0,\theta_0,r_0)$ in the first three equations of system \eqref{cns10}.
	\begin{equation}
	\label{startingsystem}
	\left\{	
	\begin{aligned}
	&[P_{0}]_t+\frac{C_v+1}{C_v}(P_0+1)\rho_0\left(r_0^2[u_{0}]\right)_x=R_1,\\[2mm]
	&[u_{0}]_t+r_0^2[P_{0}]_x=R_2,\\[2mm]
	&[s_{0}]_t=R_3.
	\end{aligned}	
	\right.
	\end{equation}
	with initial and boundary conditions that
	\begin{alignat}{2} \label{0initial}
	([P_{0}],[u_{0}],[s_{0}])(0,x)&=
	(P_0-1,u_0,s_0-1)(x),
	\qquad && x\in\mathbb{I},\\[1mm]               \label{0bdry}
	[u_{0}](t,0)&=[u_{0}](t,1)=0,\qquad&& t\geq 0.
	\end{alignat}
Let $\left([P_{0}],[u_{0}],[s_{0}]\right)$ be the solution of above system and the good in this is that instead of finding $\left([P_{0}],[u_{0}],[s_{0}]\right)\in H^{m+1}([0,T]\times\mathbb{I})$,
which satisfy the initial conditions \eqref{initialconditions}, we only need to choose functions $\left(R_1(t,x),R_2(t,x),R_3(t,x)\right)\in H^m([0,T]\times\mathbb{I})$ such that
\begin{equation}\label{initialcompatible}
	\left(\partial_t^k[P_{0}],\partial_t^k[u_{0}],\partial_t^k[s_{0}]\right)(0,x)=\left(\partial_t^kP(0,x),\partial_t^ku(0,x),\partial_t^ks(0,x)\right),\quad 1\leq k\leq m,
\end{equation}
If this can be achieved, the compatible conditions for above linear system is naturally met and with Theorem \ref{linearizedsolution} we can conclude that there exist such $\left([P_{0}],[u_{0}],[s_{0}]\right)\in X_m([0,T];\mathbb{I})$. In order to have \eqref{initialcompatible}, let
\begin{equation}\label{structure}
\left\{	\begin{aligned}
	&\partial_t^kR_1(0,x)=\partial_t^{k+1}P(0,x)+\frac{C_v+1}{C_v}(P_0+1)\rho_0\left(r_0^2\partial_t^{k}u(0,x)\right)_x,\\
	&\partial_t^kR_2(0,x)=\partial_t^{k+1}u(0,x)+r_0^2\left(\partial_t^{k}P(0,x)\right)_x,\\[1.5mm]
	&\partial_t^kR_3(0,x)=\partial_t^{k+1}s(0,x),
	\end{aligned}\right.
\end{equation}
for $0\leq k\leq m$. We can see although $\partial_t^{k+1}P(0,x),\partial_t^{k+1}u(0,x)\in H^{m-k-1}(\mathbb{I})$, due to the structure \eqref{structure} \ep{depending on the system \eqref{startingsystem}},  $\partial_t^kR_i(0,x)\in H^{m-k}(\mathbb{I})\subset H^{m-k-\frac 12}(\mathbb{I})$ \ep{$i=1,2,3$}. Now with Lemma \ref{giveninitialvalues}, we can prove such $R_i(t,x)\in H^m([0,T]\times\mathbb{I})\ (i=1,2,3)$ do exist and the proof of Lemma \ref{lemma1} completes.
\end{proof}
\end{lemma}
In the next two lemmas, we aim to prove the solution sequence generated by the iteration process is convergent in a certain function space. With the classical method \ep{see \cite[p.~34--46]{MR0748308}}, it is converted into proving two simple facts about the sequence, boundedness in the high norm and contraction in the low norm.
\begin{lemma}[boundedness in the high norm]\label{boundedness}
	There are sufficiently small constants $T_1$ and $K_3\{\epsilon_0\}$ such that $\Big(\left[P_k\right], \left[u_k\right], \left[s_k\right], \left[q_k\right],\left([r_{k-1}]^2[q_k]\right)_x\Big)
	\in B_m\left([0,T_1
	];\mathbb{I};K_3\{\epsilon_0\}\right)$. Furthermore, $\partial_t^i[u_k]=0$ on $\partial\mathbb{I}$ and
	\begin{align*}
		\left(\partial_t^i\left[P_k\right], \partial_t^i\left[u_k\right], \partial_t^i\left[s_k\right]\right)(0,x)
		=\left(\partial_t^iP(0,x),\partial_t^iu(0,x),\partial_t^is(0,x)\right),\qquad
		i=0,\cdots,m.
	\end{align*}
	\begin{proof}
			   Since $R_i(t,x)\in H^m([0,T]\times\mathbb{I})\ (i=1,2,3)$, let's say $\vertiii{R_i(t,x)}_{m,T}\leq K_0$ and according to \eqref{structure}, we have $\vertiii{\left(R_2,R_3\right)(0)}^2_m\leq C_4\{\epsilon_0\}\epsilon_0$. Now we choose a suitable $T_2$, $K_1\{\epsilon_0\}$ and $K_2\{\epsilon_0\}$ such that for any $T'$ that $0<T'\leq T_2\leq T$, our starting point $\left([P_{0}],[u_{0}],[s_{0}]\right)$ satisfy
			   \begin{align}\label{space0}
			   	\left([P_{0}],[u_{0}],[s_{0}]\right)\in A_m\left([0,T'];\mathbb{I};K_1\{\epsilon_0\}\right)\textstyle\bigcap B_m\left([0,T'];\mathbb{I};K_2\{\epsilon_0\}\right).
			   \end{align}
			   To achieve that, we apply the similar estimates to system \eqref{startingsystem} as what we did in \eqref{put}, \eqref{puv} and \eqref{sv}. Then we have
	\begin{align*}
	&\vertiii{([P_0],[u_0])(t)}_{m,tan}^2+\vertiii{([P_0],[u_0])(t)}_{m-1}^2
	\leq C_{6}\{\epsilon_0\}\cdot\exp\Big\{{t\cdot C_{8}\{\epsilon_0\}}\Big\}
	\left[\vertiii{(P,u)(0)}_{m-1}^2\right.
	\\[-1mm]
	&\left.\qquad\quad\ +
	\vertiii{(P,u)(0)}_{m,tan}^2+
	\int_{0}^{t} \vertiii{R_1(s)}_{m}^2{\rm d}s\right]
	\leq C_{6}\{\epsilon_0\}\cdot\exp\Big\{{t\cdot C_{8}\{\epsilon_0\}}\Big\}
	\left[2C_1\{\epsilon_0\}\epsilon_0+T\cdot K_0^2\right],
	\\[0.5mm]\notag
	&\vertiii{[s_0](t)}^2_{m}\lesssim
	\left\|s_0\right\|^2_{H^m}+
	\vertiii{R_3(t)}^2_{m-1}+
	\int_{0}^t\vertiii{R_3(\tau)}^2_{m}\ {\rm d}\tau
	\lesssim\epsilon_0+
	\Big(C_{4}\{\epsilon_0\}\epsilon_0
	+T\cdot K_0^2\Big)
	+T\cdot K_0^2.\\[2mm]
	&\vertiii{[s_0](t)}^2_{m,tan}\!+\vertiii{[s_0](t)}^2_{m-1}
	\leq\vertiii{R_3(t)}^2_{m-1}\!+\vertiii{[s_0](t)}^2_{m}
	\leq\epsilon_0+
	2\Big(C_{4}\{\epsilon_0\}\epsilon_0
	+T\cdot K_0^2\Big)\!
	+T\cdot K_0^2,\\[3mm]
	&\vertiii{\left([P_0],[u_0]\right)(t)}_m\leq C_{10}\{\epsilon_0\}\left[\vertiii{\left([P_0],[u_0]\right)(t)}_{m-1}
	+\vertiii{\left([P_0],[u_0]\right)(t)}_{m,tan}+\vertiii{R_2(t)}_{m-1}\right]\\[-1mm]
	&\qquad\qquad\qquad\ \ \leq C_{10}\{\epsilon_0\}\left[\vertiii{\left([P_0],[u_0]\right)(t)}_{m-1}
	+\vertiii{\left([P_0],[u_0]\right)(t)}_{m,tan}+\sqrt{C_{4}\{\epsilon_0\}\epsilon_0+T\cdot K_0^2}\right].
	\end{align*}	
		For any fixed $\epsilon_0$, denote
		\begin{align*}
			&G_{4}^*=\max\Big\{G_{4}\{\epsilon_0\},C_{4}\{\epsilon_0\}\Big\},\quad G_{6}^*=\max\Big\{G_{6}\{\epsilon_0\},C_{6}\{\epsilon_0\}\Big\},\\ &G_{10}^*=\max\Big\{G_{10}\big\{\epsilon_0,1,K_1\{\epsilon_0\}\big\},C_{10}\{\epsilon_0\}\Big\}.
		\end{align*}
		Now we define
		\begin{align*}
			&K_1\{\epsilon_0\}=2\big(G_{6}^*+1\big)\big(2C_1\{\epsilon_0\}+1\big)\epsilon_0+2\epsilon_0+2\big(G_{2}\{\epsilon_0\}+1\big)\big(G_{4}^*+1\big)\epsilon_0,\\[3mm]
		    &K_2\{\epsilon_0\}=\frac{1}{1-a}G_{10}^*\Big(2+\left(C_2\{\epsilon_0\}\right)^a\Big)
		   \left[\sqrt{2K_1\{\epsilon_0\}}
			+\sqrt{\left(G_{4}^*+1\right)\epsilon_0}\right]\\[-1.5mm]
			&\quad+\left(G_{10}^*\left[\sqrt{2K_1\{\epsilon_0\}}
			+\sqrt{\left(G_{4}^*+1\right)\epsilon_0 }\ \right]\right)^{\frac{1}{1-a}}
			+\frac{1}{1-a}\sqrt{2\epsilon_0+\left(G_{2}\{\epsilon_0\}+1\right)\left(G_{4}^*+1\right)\epsilon_0},\\[3mm]
			&K_3\{\epsilon_0\}=\left(\sqrt{G_1\big\{\epsilon_0,1,K_1\{\epsilon_0\}\big\}}+1\right)K_2\{\epsilon_0\}.
		\end{align*}

	   Since $T_2$ is small, without loss of generality we assume $T_2<1$. Furthermore, we set $T_2\cdot K_0^2\leq\epsilon_0$ and $\exp\Big\{{T_2\cdot C_{8}\big\{\epsilon_0,1,K_2\{\epsilon_0\}\big\}}\Big\}<2$. It's easy to check that
	   \begin{align*}
	   	&\vertiii{([P_0],[u_0],[s_0])(t)}_{m-1}^2+ \vertiii{([P_0],[u_0],[s_0])(t)}_{m,tan}^2\\[1.5mm]
	   	&\qquad\qquad\qquad\qquad\qquad\quad\ \,\leq 2\big(C_{4}\{\epsilon_0\}+1\big)\big(2C_1\{\epsilon_0\}+1\big)\epsilon_0+2\epsilon_0+2\big(C_{4}\{\epsilon_0\}+1\big)\epsilon_0
	   	\leq K_1\{\epsilon_0\},\\[4mm]
	   	&\vertiii{\left([P_0],[u_0]\right)(t)}_m+\vertiii{[s_0](t)}_{m}\\
	   	&\qquad\qquad\qquad\leq C_{10}\{\epsilon_0\}
	   	\left(\sqrt{2K_1\{\epsilon_0\}}
	   	+\sqrt{\left(C_{4}\{\epsilon_0\}+1\right)\epsilon_0 }\right)
	   	+\sqrt{2\epsilon_0+\left(C_{4}\{\epsilon_0\}+1\right)\epsilon_0}\leq K_2\{\epsilon_0\},
	   \end{align*}
	   and \eqref{space0} stands.
	
	    Next, we prove that there exists a $T_1$ such that for any $T'$ satisfying $0<T'\leq T_1\leq T_2$ if
	    \begin{align}\notag
	    	&\left([P_{k}],[u_{k}],[s_{k}]\right)\in A_m\left([0,T_1];\mathbb{I};K_1\{\epsilon_0\}\right)\textstyle\bigcap B_m\left([0,T_1];\mathbb{I};K_2\{\epsilon_0\}\right),\\\label{initialt0}
	    	&\left(\partial_t^i\left[P_k\right], \partial_t^i\left[u_k\right], \partial_t^i\left[s_k\right]\right)(0,x)
	    	=\left(\partial_t^iP,\partial_t^iu,\partial_t^is\right)(0),\qquad
	    	i=0,\cdots,m.
	    \end{align}
	    then
	    \begin{align*}
	    	&\Big([P_{k+1}],[u_{k+1}],[s_{k+1}],[q_{k+1}],\left([r_k]^2[q_{k+1}]\right)_x\Big)\in B_m\left([0,T_1];\mathbb{I};K_3\{\epsilon_0\}\right),\\
	    	&\left(\partial_t^i\left[P_{k+1}\right], \partial_t^i\left[u_{k+1}\right], \partial_t^i\left[s_{k+1}\right]\right)(0,x)
	    	=\left(\partial_t^iP,\partial_t^iu,\partial_t^is\right)(0),\qquad
	    	i=0,\cdots,m.
	    \end{align*}
	     First, we can see it from the definition of $\left(\partial_t^iP(0,x),\partial_t^iu(0,x),\partial_t^is(0,x)\right)$, conditions \eqref{initialt0} and the structure of system \eqref{cns10} that
	     \begin{equation*}
	     	\left(\partial_t^i\left[P_{k+1}\right], \partial_t^i\left[u_{k+1}\right], \partial_t^i\left[s_{k+1}\right]\right)(0,x)
	     	=\left(\partial_t^iP(0,x),\partial_t^iu(0,x),\partial_t^is(0,x)\right).
	     \end{equation*}
	     Second, we choose $T_1$ is small enough such that $|\bar{r}|_{L^\infty}$ has a positive lower bound, $T_1\cdot K_2\{\epsilon_0\}<\frac 12$, $T_1\cdot G_{7}\big\{\epsilon_0,1,K_1\{\epsilon_0\}\big\}<1$, $\exp\Big\{{T_1\cdot G_{8}\big\{\epsilon_0,1,K_2\{\epsilon_0\}\big\}}\Big\}<2$, $T_1\cdot G_{3}\big\{K_2\{\epsilon_0\}\big\}<1$,
	   \begin{align*}
	   T_1\cdot& G_{5}\big\{K_2\{\epsilon_0\}\big\}G_1\big\{\epsilon_0,1,K_1\{\epsilon_0\}\big\}\left(K_2\{\epsilon_0\}\right)^2<\epsilon_0,\\[2mm]
	    	G_{9}\left\{\epsilon_0,1,K_2\{\epsilon_0\}\right\}&\int_{0}^{t} \vertiii{(\bar{r}^2q)_x(s)}_{m}^2{\rm d}s\leq
	    	T_1\cdot G_{9}
	    	G_1\big\{\epsilon_0,1,K_1\{\epsilon_0\}\big\}\left(K_2\{\epsilon_0\}\right)^2\leq \epsilon_0,\\
	   \vertiii{(\bar{r}^2q)_x(t)}_{m-1}^2&
	   \mathrel{\overset{\makebox[0pt]{\mbox{\scriptsize\sffamily  \eqref{smtan}}}}{\leq}}
	   \
	   G_{4}\{\epsilon_0\}\epsilon_0
	   +T_1\cdot G_1\big\{\epsilon_0,1,K_1\{\epsilon_0\}\big\}\left(K_2\{\epsilon_0\}\right)^2
	   \leq \left(G_{4}\{\epsilon_0\}+1\right)\epsilon_0.
	   \end{align*}
	   Then we have from \eqref{put}, \eqref{puv}, \eqref{smtan}, \eqref{sv} and \eqref{barq} that
	   \begin{align*}
	   	&\vertiii{([P_{k+1}],[u_{k+1}],[s_{k+1}])(t)}_{m,tan}^2+\vertiii{([P_{k+1}],[u_{k+1}],[s_{k+1}])(t)}_{m-1}^2\\[1mm]
	   	&\qquad\qquad\qquad
	   	\leq2\big(G_{6}\{\epsilon_0\}+1\big)\big(2C_1\{\epsilon_0\}+1\big)\epsilon_0+2\epsilon_0+2\big(G_{2}\{\epsilon_0\}+1\big)\big(G_{4}\{\epsilon_0\}+1\big)\epsilon_0\leq K_1\{\epsilon_0\},\\[2mm]
	   	&\vertiii{\left([P_{k+1}],[u_{k+1}]\right)(t)}_m+\vertiii{[s_{k+1}](t)}_{m}\\[-0.5mm]
	   	&\qquad\qquad\qquad\mathrel{\overset{\makebox[0pt]{\mbox{\scriptsize\sffamily  \eqref{barr}}}}{\leq}}
	   	G_{10}\big\{\epsilon_0,1,K_1\{\epsilon_0\}\big\}\vertiii{\left([P_{k}],[u_{k}],[s_{k}]\right)}_{m,T}^a\left[\sqrt{2K_1\{\epsilon_0\}}+\sqrt{\left(G_{4}\{\epsilon_0\}+1\right)\epsilon_0}\right]\\
	   	&\qquad\qquad\qquad\quad\ +
	   	G_{10}\big\{\epsilon_0,1,K_1\{\epsilon_0\}\big\}
	   	\Big(2+\left(C_2\{\epsilon_0\}\right)^a\Big)\left[\sqrt{2K_1\{\epsilon_0\}}+\sqrt{\left(G_{4}\{\epsilon_0\}+1\right)\epsilon_0}\right]\\
	   	&\qquad\qquad\qquad\quad\ +\sqrt{2\epsilon_0+\left(G_{2}\{\epsilon_0\}+1\right)\left(G_{4}\{\epsilon_0\}+1\right)\epsilon_0}\\[-1mm]
	   	&\qquad\ \ \mathrel{\overset{\raise1mm\hbox{\scriptsize\sffamily  Young's inequality}}{\leq}}\!\!\!\!\!\!\!\!
	   	aK_2\{\epsilon_0\}
	   	+(1-a)\left(G_{10}\big\{\epsilon_0,1,K_1\{\epsilon_0\}\big\}
	   	\left[\sqrt{2K_1\{\epsilon_0\}}+\sqrt{\left(G_{4}\{\epsilon_0\}+1\right)\epsilon_0}\right]\right)^\frac{1}{1-a}\\
	   	&\qquad\qquad\qquad\quad\
	   	+G_{10}\big\{\epsilon_0,1,K_1\{\epsilon_0\}\big\}
	   	\left(2+\left(C_2\{\epsilon_0\}\right)^a\right)\left[\sqrt{2K_1\{\epsilon_0\}}+\sqrt{\left(G_{4}\{\epsilon_0\}+1\right)\epsilon_0}\right]\\
	   	&\qquad\qquad\qquad\quad\
	   	+\sqrt{2\epsilon_0+\left(G_{2}\{\epsilon_0\}+1\right)\left(G_{4}\{\epsilon_0\}+1\right)\epsilon_0}\\[1mm]
	   	&\qquad\qquad\qquad
	   	\leq aK_2\{\epsilon_0\}+(1-a)K_2\{\epsilon_0\}=K_2\{\epsilon_0\},\\[2mm]
	   	&\vertiii{([P_{k+1}],[u_{k+1}],[s_{k+1}])(t)}_{m}+\vertiii{\Big([q_{k+1}],\left([r_k]^2[q_{k+1}]\right)_x\Big)(t)}_{m}\\[-1.5mm]
	   	&\qquad\qquad\qquad\qquad\qquad\qquad\qquad\qquad\qquad\quad
	   	\mathrel{\overset{\makebox[0pt]{\mbox{\scriptsize\sffamily  \eqref{barq}}}}{\leq}}
	   	K_2\{\epsilon_0\}
	   	+\sqrt{G_1\big\{\epsilon_0,1,K_1\{\epsilon_0\}\big\}}K_2\{\epsilon_0\}
	   	=K_3\{\epsilon_0\}.
	   \end{align*}
	\end{proof}
\end{lemma}
\begin{lemma}[contraction in low norm]\label{contraction}
	A $T_0>0$ can be found such that if $T\leq T_0$, there exist functions $\left(P,u,s,q\right)\in C\left\{[0,T],L^2(\mathbb{I})\right\}$ satisfy
	\begin{align}
		\big(\,[P_k],[u_k],[s_k],[q_k]\,\big)\rightarrow\left(P,u,s,q\right)\ \text{in}\ \ C\{[0,T],L^2(\mathbb{I})\},\quad \text{as}\ k\rightarrow\infty.
	\end{align}
	\begin{proof}
		If system \eqref{cns10} is denoted as $E_{k+1}$, we subtract $E_k$ from $E_{k+1}$ and get the equations for $\left(\widetilde{[P_{k}]},\widetilde{[u_{k}]},\widetilde{[s_{k}]},\widetilde{[q_{k}]}\right):=\Big([P_{k+1}]-[P_{k}],\ [u_{k+1}]-[u_{k}],\ [s_{k+1}]-[s_{k}],\ [q_{k+1}]-[q_{k}]\Big)$ that
		\begin{equation}
		\label{cns11}
		\left\{	
		\begin{aligned}
		&\widetilde{[P_{k}]}_t+\frac{C_v+1}{C_v}([P_k]+1)[\rho_k]\widetilde{[U_{k}]}_x=\mathbb{R}_n^1,\\[2mm]
		&\widetilde{[u_{k}]}_t+[r_k]^2\widetilde{[P_{k}]}_x=\mathbb{R}_n^2,\\[2mm]
		&\widetilde{[s_{k}]}_t=\mathbb{R}_n^3,\\[2mm]
		&\frac{\widetilde{[q_{k}]}}{[r_k]^2[\rho_k]}-\Big([\rho_k]\widetilde{[Q_k]}_{x}\Big)_x
		=\mathbb{R}_n^4.
		\end{aligned}	
		\right.
		\end{equation}
		Here,
		\begin{align}\label{UkQk}
			&\widetilde{[U_{k}]}:=[r_k]^2[u_{k+1}]-[r_{k-1}]^2[u_{k}],\qquad\
			\widetilde{[Q_{k}]}:=[r_k]^2[q_{k+1}]-[r_{k-1}]^2[q_{k}],\\[3mm]\notag
			&\mathbb{R}_n^1=-\left(\frac{C_v+1}{C_v}([P_k]+1)[\rho_k]
			-\frac{C_v+1}{C_v}([P_{k-1}]+1)[\rho_{k-1}]\right)\left([r_{k-1}]^2[u_{k}]\right)_x\\[-1.5mm]\notag
			&\qquad\
			-\frac{1}{C_v}[\rho_k]\widetilde{[Q_k]}_x
			 -\Big([\rho_k]-[\rho_{k-1}]\Big)\frac{\left([r_{k-1}]^2[q_{k}]\right)_x}{C_v},\\[2mm]\notag
			&\mathbb{R}_n^2=-\left([r_k]^2-[r_{k-1}]^2\right)[P_k]_x,\qquad
			\mathbb{R}_n^3=-\frac{\widetilde{[Q_k]}_x}{[\theta_k]}-\left(\frac{1}{[\theta_k]}-\frac{1}{[\theta_{k-1}]}\right)\left([r_{k-1}]^2[q_{k}]\right)_x,\\[2mm]\notag
			&\mathbb{R}_n^4=-\left(\frac{1}{[r_k]^2[\rho_k]}-\frac{1}{[r_{k-1}]^2[\rho_{k-1}]}\right)[q_k]
			+\Big(\big([\rho_k]-[\rho_{k-1}]\big)\left([r_{k-1}]^2[q_{k}]\right)_x\Big)_x\\\notag
			&\qquad\
			-\frac{4[\theta_k]^3}{(C_v+1)[\rho_k]}\Big(\widetilde{[P_{k-1}]}_x+\left([P_{k}]+1\right)\widetilde{[s_{k-1}]}_x\Big)
			-\left(\frac{4[\theta_k]^3}{(C_v+1)[\rho_k]}-\frac{4[\theta_{k-1}]^3}{(C_v+1)[\rho_{k-1}]}\right)[P_{k-1}]_x\\\notag
			&\qquad\
			-\left(\frac{4[\theta_k]^3}{(C_v+1)[\rho_k]}\left([P_{k}]+1\right)-\frac{4[\theta_{k-1}]^3}{(C_v+1)[\rho_{k-1}]}\left([P_{k-1}]+1\right)\right)[s_{k-1}]_x.
		\end{align}
		Here, just like what we did in \eqref{barq}, the reason why we treat $\widetilde{[U_{k}]}$ and $\widetilde{[Q_{k}]}$ as a whole is to avoid $\widetilde{[u_{k-1}]}_x$ emerging from $[r_k]_x-[r_{k-1}]_x$, which would make the order of derivatives not balanced in \eqref{contractionPus}. Then we compute $\eqref{cns11}_1\times\frac{C_v}{(C_v+1)[\rho_k]}\widetilde{[P_{k}]}+\eqref{cns11}_2\times\frac{1}{[r_k]^2}\widetilde{[U_k]}+\eqref{cns11}_3\times\widetilde{[s_{k}]}$ and integrate the resultant equality over $\mathbb{I}$, which gives
		\begin{align}\notag
		\frac{{\rm d}}{{\rm d}t}\int_{\mathbb{I}}&\left(\frac{C_v}{(C_v+1)[\rho_k]}\widetilde{[P_{k}]}^2
		+\frac{\widetilde{[u_{k}]}^2}{2}
		+\frac{[r_k]\widetilde{[r_{k-1}]}[u_k]\widetilde{[u_{k}]} }{[r_k]^2}
		+\frac{\widetilde{[r_{k-1}]}[r_{k-1}][u_k]\widetilde{[u_{k}]} }{[r_k]^2}
		+\frac{\widetilde{[s_{k}]}^2}{2}
		\right){\rm d}x\\\label{Pusk}
		&\lesssim  \left\|\left(\widetilde{[P_{k}]},\widetilde{[u_{k}]},\widetilde{[s_{k}]}\right)(t)\right\|^2
		+\left\|\left(\widetilde{[Q_{k}]}_x,\widetilde{[P_{k-1}]},\widetilde{[s_{k-1}]}\right)(t)\right\|^2
		+T\cdot\sup_{t\in[0,T]}\left\|\widetilde{[u_{k-1}]}(t)\right\|^2.
		\end{align}
		Integrate \eqref{Pusk} over $[0,t]$ \ep{$0\leq t\leq T$} and we have
		\begin{align*}
			 \left\|\left(\widetilde{[P_{k}]},\widetilde{[u_{k}]},\widetilde{[s_{k}]}\right)(t)\right\|^2
			 \lesssim& \int_{0}^t\left\|\left(\widetilde{[P_{k}]},\widetilde{[u_{k}]},\widetilde{[s_{k}]}\right)(\tau)\right\|^2{\rm d}\tau
			 +\int_0^t\left\|\left(\widetilde{[Q_{k}]}_x,\widetilde{[P_{k-1}]},\widetilde{[s_{k-1}]}\right)(\tau)\right\|^2{\rm d}\tau\\
			 &+t\cdot\sup_{t\in[0,T]}\left\|\widetilde{[u_{k-1}]}(t)\right\|^2.
			 \qquad\quad\text{\ep{$T<1$ set in Lemma \ref{boundedness}}}
		\end{align*}
		With $Gr\ddot{o}nwall's$ inequality, we have for some positive constant $C$ that
		\begin{align}\label{Puskl}
		\left\|\left(\widetilde{[P_{k}]},\widetilde{[u_{k}]},\widetilde{[s_{k}]}\right)(t)\right\|^2\leq CTe^{CT}\sup_{t\in[0,T]}\left\|\left(\widetilde{[Q_{k}]}_x,\widetilde{[P_{k-1}]},\widetilde{[u_{k-1}]},\widetilde{[s_{k-1}]}\right)(t)\right\|^2.
		\end{align}
		Now we compute $\eqref{cns11}_4\times\widetilde{[Q_k]}$, integrate the resultant equality over $\mathbb{I}$ and we have for for some positive constant $C$ that
		\begin{align}\label{qk}
			\left\|\widetilde{[q_{k}]}\right\|^2
			+\left\|\widetilde{[Q_{k}]}_x\right\|^2\leq C\left\|\left(\widetilde{[P_{k-1}]},\widetilde{[u_{k-1}]},\widetilde{[s_{k-1}]}\right)\right\|^2.
		\end{align}
		Substitute \eqref{qk} into \eqref{Puskl} and we have
		\begin{align}\label{contractionPus}
			\left\|\left(\widetilde{[P_{k}]},\widetilde{[u_{k}]},\widetilde{[s_{k}]}\right)(t)\right\|^2\leq CTe^{CT}\sup_{t\in[0,T]}\left\|\left(\widetilde{[P_{k-1}]},\widetilde{[u_{k-1}]},\widetilde{[s_{k-1}]}\right)(t)\right\|^2.
		\end{align}
		We can see there exist a fixed constant $0<\gamma<1$ and a small $T_0>0$ such that if $T\leq T_0$, then $CTe^{CT}\leq\gamma$. According to Banach fixed-point theorem, there exist  $\left(P,u,s\right)\in C\{[0,T],L^2(\mathbb{I})\}$ such that
			\begin{align}\label{Puslimit}
		\big(\,[P_k],[u_k],[s_k]\,\big)\rightarrow\left(P,u,s\right)\ \text{in}\ \ C\{[0,T],L^2(\mathbb{I})\},\quad \text{as}\ k\rightarrow\infty.
		\end{align}
With \eqref{qk} and \eqref{contractionPus}, we have for some $q\in C\{[0,T],L^2(\mathbb{I})\}$	that
\begin{align}\label{qlimit}
	[q_k]\rightarrow q\ \text{in}\ \ C\{[0,T],L^2(\mathbb{I})\},\quad \text{as}\ k\rightarrow\infty,
\end{align}
according to Cauchy's Criterion.			
 	\end{proof}
\end{lemma}
Since we have proved Lemma \ref{boundedness} and \ref{contraction}, by utilizing the Sobolev space interpolation inequality
\begin{align}\label{interpolation}
	\left\|f\right\|_{H^s}\leq C_m\left\|f\right\|_{H^m}^{\frac{s}{m}}\left\|f\right\|^{1-\frac{s}{m}},
\end{align}
we have for any $0<s<m$ that
\begin{multline}\label{l2contraction}
	\left\|\big(\,[P_k]-[P_l],\,[u_k]-[u_l],\,[s_k]-[s_l],\,[q_k]-[q_l]\,\big)\right\|_{H^s}\\
	\leq CK_3\{\epsilon_0\}\left\|\big(\,[P_k]-[P_l],\,[u_k]-[u_l],\,[s_k]-[s_l],\,[q_k]-[q_l]\,\big)\right\|^{1-\frac{s}{m}}.
\end{multline}
From \eqref{Puslimit}, \eqref{qlimit} and \eqref{l2contraction}, we conclude that
\begin{equation*}
	\lim_{k\rightarrow+\infty}\sup_{t\in[0,T_0]}\left\|\big(\,[P_k]-P,\,[u_k]-u,\,[s_k]-s,\,[q_k]-q\,\big)(t)\right\|_{H^{s}}=0.
\end{equation*}
If we choose $s>\frac{1}{2}+1$, Sobolev inequality implies that
\begin{equation}\label{Pusq}
	\big(\,[P_k],\,[u_k],\,[s_k],\,[q_k]\,\big)\rightarrow\big(\,P,\,u,\,s,\,q\,\big)\quad\text{in}\quad C\{[0,T_0],C^1(\mathbb{I})\}.
\end{equation}
From the structure of first three equation in \eqref{cns10}, we can see
\begin{equation*}
		\big(\,[P_k]_t,\,[u_k]_t,\,[s_k]_t\,\big)\rightarrow\big(\,P_t,\,u_t,\,s_t\,\big)\quad\text{in}\quad C\{[0,T_0],C(\mathbb{I})\},
\end{equation*}
and $\big(\,P,\,u,\,s\,\big)$ belong to $C^1\{[0,T_0]\times\mathbb{I}\}$. So the first three equation of \eqref{cns4} hold in classical sense. So \eqref{r_eq} holds. We have from \eqref{Pusq} that
\begin{equation*}
	\left([r_{k-1}]^2[q_k]\right)_x\rightarrow(r^2q)_x\quad\text{in}\quad C\{[0,T_0],L^2(\mathbb{I})\},
\end{equation*}
Together with \eqref{qk}, \eqref{interpolation}, it can be derived that
\begin{align}\label{r2q}
	\left([r_{k-1}]^2[q_k]\right)_x\rightarrow(r^2q)_x\quad\text{in}\quad C\{[0,T_0],C^1(\mathbb{I})\}.
\end{align}
With \eqref{r_eq}, \eqref{Pusq} and \eqref{r2q}, we have $q_x\in C\{[0,T_0],C^1(\mathbb{I})\}$, which indicate that the last equation of \eqref{cns4} also holds in classical sense.

From the technique in \cite[p.~34--46]{MR0748308}, we can raise the regularity of solution such that
\begin{equation*}
	(P,u,s,q,q_x)\in X_m([0,T_0],\mathbb{I}).
\end{equation*}
We omit the proof here. One can refer to the proof of Theorem $2.1(a)$ and $2.1(b)$ in \cite[p.~34--46]{MR0748308}. Now the proof of local existence completes. To obtain the global-in-time classical solution, first we denote $K_3^*\{\epsilon_0\}=\max\big\{K_3\{\epsilon_0\},K_3\{V_0\}\big\}$ and set $\epsilon_0$ is small enough such that $K_3^*\{\epsilon_0\}\leq\epsilon$. \ep{Here $\epsilon$ is defined in Proposition \ref{a priori} to ensure that the a priori estimates stand.} Second we set $T\leq T_0$. \ep{Here $T_0$ is fixed, since $\epsilon_0$ is already set.} And we have the classical solution in $[0,T]$. Third, with the a priori estimates, apply the method of continuity like what is stated before Section \ref{a priori estimates}. Finally, we finish the proof of Theorem \ref{thm}.

\section{Acknowledgements.} The research is supported by a grant from National Natural Science Foundation of China under contract No.12221001 and a grant from Science and Technology Department of Hubei Province under contract No.2020DFH002.

\bibliographystyle{abbrvnat}

\bibliography{references}

\end{document}